\documentclass[final]{dmtcs-episciences}
\usepackage{tikz}
\usetikzlibrary{arrows}

\usepackage[ansinew]{inputenc}          
\usepackage{amssymb,amsthm,amsmath}     
\usepackage{subfigure}
\usepackage{a4wide}

\theoremstyle{plain}
\newtheorem{theorem}{Theorem}
\newtheorem{con}[theorem]{Conjecture}
\newtheorem{lemma}[theorem]{Lemma}

\newtheorem{proposition}[theorem]{Proposition}
\theoremstyle{definition}

\newtheorem{example}[theorem]{Example}
\newtheorem{remark}[theorem]{Remark}

\newcommand{\Z}{\mathbb{Z}}

\newcommand{\ID}{\gamma^{ID}}
\newcommand{\LD}{\gamma^{LD}}

\newcommand{\ms}{\bar{s}}

\author{Ville Junnila\affiliationmark{1}
  \and Tero Laihonen\affiliationmark{1}
  \and Gabrielle Paris\affiliationmark{2}\thanks{Supported by the ANR-14-CE25-0006 project of the French National Research Agency}}
\title[Solving conjectures on codes for location in circulant graphs]{Solving Two Conjectures regarding Codes for Location in Circulant Graphs}
\affiliation{
  Department of Mathematics and Statistics, University of Turku, Finland\\
  LIRIS, University of Lyon, France}
\received{2017-10-3}
\revised{2018-12-17}
\accepted{2018-12-17}

\begin{document}
\publicationdetails{21}{2019}{3}{2}{3973}
\maketitle

\begin{abstract}
Identifying and locating-dominating codes have been widely studied in circulant graphs of type $C_n(1,2, \ldots, r)$, which can also be viewed as power graphs of cycles. Recently, Ghebleh and Niepel (2013) considered identification and location-domination in the circulant graphs $C_n(1,3)$. They showed that the smallest cardinality of a locating-dominating code in $C_n(1,3)$ is at least $\lceil n/3 \rceil$ and at most $\lceil n/3 \rceil + 1$ for all $n \geq 9$. Moreover, they proved that the lower bound is strict when $n \equiv 0, 1, 4 \pmod{6}$ and conjectured that the lower bound can be increased by one for other $n$. In this paper, we prove their conjecture. Similarly, they showed that the smallest cardinality of an identifying code in $C_n(1,3)$ is at least $\lceil 4n/11 \rceil$ and at most $\lceil 4n/11 \rceil + 1$ for all $n \geq 11$. Furthermore, they proved that the lower bound is attained for most of the lengths $n$ and conjectured that in the rest of the cases the lower bound can improved by one. This conjecture is also proved in the paper. The proofs of the conjectures are based on a novel approach which, instead of making use of the local properties of the graphs as is usual to identification and location-domination, also manages to take advantage of the global properties of the codes and the underlying graphs.
\end{abstract}



\section{Introduction}

Let $G = (V, E)$ be a simple, undirected graph with the vertex set $V$ and the edge set $E$. The \emph{open neighbourhood} $N(G;u)$ of $u \in V$ consists of the vertices adjacent to $u$, i.e., $N(G;u) = \{ v \in V \ | \ uv \in E \}$. The \emph{closed neighbourhood} $N[G;u]$ of $u \in V$ is defined as $N[G;u] = N(u) \cup \{u\}$. Regarding the open and closed neighbourhoods, if the underlying graph is known from the context, then we can simply write $N(G;u) = N(u)$ and $N[G;u] = N[u]$. A nonempty subset $C \subseteq V$ is called a \emph{code}, and its elements are called \emph{codewords}. The \emph{identifying set} (or the $I$\emph{-set} or the \emph{identifier}) of $u$ is defined as $I(G,C;u) = N[G;u] \cap C$; if the graph $G$ or the code $C$ is known from the context, then we can again write $I(G,C;u) = I(G;u) = I(C;u) = I(u)$. 

Let $C$ be a code in $G$. A vertex $u \in V$ is \emph{covered} or \emph{dominated} by a codeword of $C$ if the identifying set $I(C;u)$ is nonempty. The code $C$ is \emph{dominating} in $G$ if all the vertices of $V$ are covered by a codeword of $C$, i.e., $|I(C;u)| \geq 1$ for all $u \in V$. The code $C$ is \emph{identifying} in $G$ if $C$ is dominating and for all distinct $u,v \in V$ we have
\[
I(C;u) \neq I(C;v) \textrm{.}
\]
The definition of identifying codes is due to Karpovsky \emph{et al.}~\cite{kcl}, and the original motivation for studying such codes comes from fault diagnosis in multiprocessor systems. The concept of locating-dominating codes is closely related to the one of identifying codes. We say that the code is \emph{locating-dominating} in $G$ if $C$ is dominating and for all distinct $u,v \in V \setminus C$ we have $I(C;u) \neq I(C;v)$.
The definition of locating-dominating codes was introduced by Slater~\cite{RS:LDnumber,S:DomLocAcyclic,S:DomandRef}. The original motivation for locating-dominating codes was based on fire and intruder alarm systems. An identifying or locating-dominating code with the smallest cardinality in a given finite graph $G$ is called \emph{optimal}. The number of codewords in an optimal identifying and locating-dominating code in a finite graph $G$ is denoted by $\ID(G)$ and $\LD(G)$, respectively.

In this paper, we focus on studying identifying and locating-dominating codes in so called circulant graphs. For the definition of circulant graphs, we first assume that $n$ and $d_1, d_2, \ldots, d_k$ are positive integers and $d_i\le n/2$ for all $i=1,\dots,k$. Then the circulant graph $C_n(d_1, d_2, \ldots, d_k)$ is defined as follows: the vertex set is $\Z_n = \{0,1, \ldots, n-1\}$ and the open neighbourhood of a vertex $u \in \Z_n$ is
\[
N(u) = \{u \pm d_1, u \pm d_2, \ldots, u \pm d_k\} \text{,}
\]
where the calculations are done modulo $n$. Previously, in~\cite{BCHLildchc,CLMcp,EJLldc,GMS:IdCyc,JLidcp,Mlldcn,RobRob,XTHicco}, identifying and locating-dominating codes have been studied in the circulant graphs $C_n(1,2, \ldots, r)$ $(r \in \Z, r \geq 1)$, which can also be viewed as power graphs of cycles of length $n$. (In Remark~\ref{RemarkManuelProblem} in the end of Section~\ref{SectionLD}, we discuss some problems occurring in the paper~\cite{Mlldcn}.) For various other papers on the subject, see the
online bibliography~\cite{lowww}. Recently, in~\cite{GNlidcn}, Ghebleh and Niepel studied identification and location-domination in $C_n(1,3)$. Regarding locating-dominating codes, they showed that $\LD(C_n(1,3)) = \lceil n/3 \rceil$ if $n \equiv 0, 1, 4 \pmod{6}$ and $\lceil n/3 \rceil \leq \LD(C_n(1,3)) \leq \lceil n/3 \rceil +1$ if $n \equiv 2, 3, 5 \pmod{6}$. Furthermore, they stated the following conjecture, which we prove in Section~\ref{SectionLD}.
\begin{con} \label{ConjectureLD}
If $n$ is an integer such that $n \geq 13$ and $n \equiv 2, 3, 5 \pmod{6}$, then we obtain that $\LD(C_n(1,3)) = \lceil n/3 \rceil+1$.
\end{con}

Concerning identifying codes, Ghebleh and Niepel~\cite{GNlidcn} claimed that $\lceil 4n/11 \rceil \leq \ID(C_n(1,3)) \leq \lceil 4n/11 \rceil +1$ if $n \equiv 8 \pmod{11}$ and $\ID(C_n(1,3)) = \lceil 4n/11 \rceil$  if $n \not\equiv 8 \pmod{11}$. However, the constructions given in their paper are erroneous as is shown in Section~\ref{SectionID}. Moreover, we prove that $\ID(C_n(1,3)) = \lceil 4n/11 \rceil + 1$ for $n \equiv 2, 5, 8 \pmod{11}$ when $n$ is large enough. Thus, we prove the following conjecture stated in~\cite{GNlidcn}.
\begin{con} \label{ConjectureID}
If $n$ is an integer such that $n \geq 19$ and $n \equiv 8 \pmod{11}$, then we have $\ID(C_n(1,3)) = \lceil 4n/11 \rceil~+1$.
\end{con}

In~\cite{GNlidcn}, Ghebleh and Niepel also stated as an open question what happens regarding identification and location-domination in $C_n(1,d)$ with $d > 3$. These questions have been considered in our papers~\cite{JLPlcg} and~\cite{JLPobclcg}.

The proofs of the lower bounds presented in~\cite{GNlidcn} are based on a concept of share. In Section~\ref{SectionLBShare}, we give the definition of share and also present the method for obtaining lower bounds based on the concept. As we shall see, the share is a local property of a graph. The proofs in ~\cite{GNlidcn} are as well based on local properties of the graph (which is typical for identification and location-domination). However, in this paper, we introduce a new approach which enables us to also make use of the global properties of the graph. This new approach is applied to identifying codes in Section~\ref{SectionID} and to locating-dominating codes in Section~\ref{SectionLD}.


\section{Lower bounds using share} \label{SectionLBShare}

Let $G = (V, E)$ be a simple, connected and undirected graph. Assume that $C$ is a code in $G$. The following concept of the share of a codeword has been introduced by Slater in \cite{S:fault-tolerant}. The \emph{share} of a codeword $c \in C$ is denoted by $s(c)$ and
defined as
\[
s(C; c) = s(c) = \sum_{u \in N[c]} \frac{1}{|I(C; u)|} \textrm{.}
\]
The notion of share proves to be useful in determining lower bounds of identifying and locating-dominating codes (as explained in the following paragraph).

Assume that $G$ is a finite graph and $D$ is a dominating code in $G$, i.e., $N[u] \cap D$ is non-empty for all $u \in V$. Then it is easy to conclude that $\sum_{c \in D} s(D;c) = |V|$. Assume further that $s(D;c) \leq \alpha$ for all $c \in D$. Then we have $|V| \leq
\alpha |D|$, which immediately implies
\[
|D| \geq \frac{1}{\alpha} |V| \textrm{.}
\]
Assume then that for any identifying code $C$ in $G$ we have $s(C;c) \leq \alpha$ for all $c \in C$. By the aforementioned observation, we then obtain the lower bound $|V|/\alpha$ for the size of an identifying code in $G$. In other words, by determining the maximum share for any identifying code, we obtain a lower bound for the minimum size of an identifying code. The same technique obviously applies also to locating-dominating codes. However, for our purposes, it is not enough to just consider the maximum share, but we need to use a more sophisticated method by determining an upper bound for the share of a codeword \emph{on average}. The averaging process is done by introducing a \emph{shifting scheme} to even out the shares among the codewords.

In~\cite{GNlidcn}, Ghebleh and Niepel obtain the lower bounds $\ID(C_n(1,3)) \geq \lceil 4n/11 \rceil$ and $\LD(C_n(1,3)) \geq \lceil n/3 \rceil$ using a similar technique. In their paper, it is shown that for any identifying code in $C_n(1,3)$ the share of a codeword is on average at most $11/4$ and for any locating-dominating at most $3$. These observations then lead to the previous lower bounds. In our paper, we study the maximum average share more carefully and show that the values $11/4$ and $3$ can be achieved only for \emph{specific patterns} of codewords and non-codewords. Then the improved lower bounds are obtained studying the global codeword distributions in the graphs based on these specific patterns. 

\section{Identifying codes in $C_n(1,3)$} \label{SectionID}

Let us first consider constructions of identifying codes in $C_n(1,3)$. As stated in the introduction, the code constructions given in~\cite{GNlidcn} are erroneous. For example, in the paper, it is claimed that for a positive integer $q$ the code $B_q = \{ 11i+j \mid 0 \leq i \leq q-1 \text{ and } j \in \{0,4,5,6\} \}$ is identifying in $C_{11q}(1,3)$. However, this is not the case since $I(B_q;0) = I(B_q;n-1) = \{0\}$. Notice that the constructions given for the lengths other than $n \equiv 0 \pmod{11}$ are based on $B_q$ and, hence, these constructions also have some problems. 
In most cases, the constructions can be fixed as is shown in Theorem~\ref{TheoremIDConstructions}. However, in Theorem~\ref{TheoremIDLBs}, we show that there does not exist an identifying code in $C_n(1,3)$ with $\lceil 4n/11 \rceil$ codewords for $n \equiv 2,5 \pmod{11}$ when $n$ is large enough. Recall that in~\cite{GNlidcn} it is claimed that such codes exist. In the following theorem, we give the general upper bounds and constructions for identifying codes in $C_n(1,3)$.
\begin{theorem} \label{TheoremIDConstructions}
Let $n$ be an integer such that $n \geq 11$. If $n \equiv 2,5,8 \pmod{11}$, then we have $\lceil 4n/11 \rceil \leq \ID(C_n(1,3)) \leq \lceil 4n/11 \rceil +1$, and otherwise $\ID(C_n(1,3)) = \lceil 4n/11 \rceil$.
\end{theorem}
\begin{proof}
Let $n$, $q$ and $r$ be integers such that $n = 11q + r$, $q \geq 0$ and $0 \leq r < 11$. Recall first that any identifying code in $C_n(1,3)$ has at least $\lceil 4n/11 \rceil $ codewords by~\cite{GNlidcn}. For the constructions, we first define a code
\[
C_q = \{ 11i+j \mid 0 \leq i \leq q-1 \text{ and } j \in \{0,1,4,5\} \} \text{.}
\]
Let then $A$ be the following set of vertices: $A = \{3, 4, \ldots, 11q-4\}$. In Table~\ref{TableIDISets}, we have listed the identifying sets $I(C_q;u)$ and their reductions modulo $11$ for all $u \in A$ depending on the remainder when $u$ is divided by $11$. Comparing the identifying sets $I(C_q;u) \pmod{11}$, we immediately observe that $I(C_q;u) \neq I(C_q;v)$ for all $u,v \in A$ and $u \not\equiv v \pmod{11}$. Moreover, if $u \equiv v \pmod{11}$ and $u \neq v$, then $I(C_q;u) \neq I(C_q;v)$ as $N[u] \cap N[v] = \emptyset$. This implies that $C_q$ is an identifying set in $C_{11q}(1,3)$ since it is straightforward to verify that $I(C_q;u)$ are also non-empty and unique for all $u \in \{0,1,2,11q-3, 11q-2, 11q-1\}$. Similarly, it can be shown that the codes given in Table~\ref{TableIDCodes} are identifying in $C_n(1,3)$. Observe that the cardinalities of the identifying codes are also given in the table. Therefore, as the cardinalities meet the ones given in the claim, the proof is concluded.
\begin{table}
\begin{center}
\begin{tabular}{c|c|c}
$u \in A \pmod{11}$ & $I(C_q;u)$  & $I(C_q;u) \pmod{11}$\\
\hline
$0$ & $\{u,u+1\}$ & $\{0,1\}$ \\
$1$ & $\{u-1,u,u+3\}$ & $\{0,1,4\}$ \\
$2$ & $\{u-1,u+3\}$ & $\{1,5\}$ \\
$3$ & $\{u-3,u+1\}$ & $\{0,4\}$ \\
$4$ & $\{u-3,u,u+1\}$ & $\{1,4,5\}$ \\
$5$ & $\{u-1,u\}$ & $\{4,5\}$ \\
$6$ & $\{u-1\}$ & $\{5\}$ \\
$7$ & $\{u-3\}$ & $\{4\}$ \\
$8$ & $\{u-3,u+3\}$ & $\{0,5\}$ \\
$9$ & $\{u+3\}$ & $\{1\}$ \\
$10$ & $\{u+1\}$ & $\{0\}$ \\
\end{tabular}
\end{center}
\caption{Identifying sets $I(C_q;u)$ and their reductions modulo $11$ for all $u \in A$} \label{TableIDISets}
\end{table}
\begin{table}
\begin{center}
\begin{tabular}{c|c|c}
$n$ & identifying code $C$  & $|C|$\\
\hline
$11q$ & $C_q$ & $4q = \lceil 4n/11 \rceil$ \\
$11q + 1$ & $C_q \cup \{11q \}$ & $4q+1 = \lceil 4n/11 \rceil$ \\
$11q + 2$ & $C_q \cup \{11q, 11q+1 \}$ & $4q+2 = \lceil 4n/11 \rceil + 1$ \\
$11q + 3$ & $C_q \cup \{11q, 11q+1 \}$ & $4q+2 = \lceil 4n/11 \rceil$ \\
$11q + 4$ & $C_q \cup \{11q, 11q+1 \}$ & $4q+2 = \lceil 4n/11 \rceil$ \\
$11q + 5$ & $C_q \cup \{11q, 11q+1, 11q+2 \}$ & $4q+3 = \lceil 4n/11 \rceil + 1$ \\
$11q + 6$ & $C_q \cup \{11q, 11q+1, 11q+2 \}$ & $4q+3 = \lceil 4n/11 \rceil$ \\
$11q + 7$ & $C_q \cup \{11q, 11q+1, 11q+3 \}$ & $4q+3 = \lceil 4n/11 \rceil$ \\
$11q + 8$ & $C_q \cup \{11q, 11q+1, 11q+2, 11q+3 \}$ & $4q+4 = \lceil 4n/11 \rceil + 1$ \\
$11q + 9$ & $C_q \cup \{11q, 11q+1, 11q+2, 11q+3 \}$ & $4q+4 = \lceil 4n/11 \rceil$ \\
$11q + 10$ & $C_q \cup \{11q, 11q+1, 11q+3, 11q+4 \}$ & $4q+4 = \lceil 4n/11 \rceil$ \\
\end{tabular}
\end{center}
\caption{Identifying codes in $C_n(1,3)$ for $n = 11q+r$ and their cardinalities} \label{TableIDCodes}
\end{table}
\end{proof}

The general constructions given in the previous theorem can be improved for certain lengths $n$. These smaller identifying codes are given in Table~\ref{TableIDSpecificCodes}. It is straightforward to verify that these codes are indeed identifying. Observe also that the codes are optimal, i.e., attain the lower bound $\lceil 4n/11 \rceil$.
\begin{table}
\begin{center}
\begin{tabular}{c|c|c}
$n$ & identifying code $C$  & $|C|$\\
\hline
13 & $\{ 0,1,4,7,8 \}$ & $\lceil 4n/11 \rceil = 5$ \\
16 & $\{ 0,1,4,7,10,11 \}$ & $\lceil 4n/11 \rceil = 6$ \\
24 & $\{ 0,1,2,6,9,10,15,16,19 \}$ & $\lceil 4n/11 \rceil = 9$ \\
27 & $\{ 0,1,2,6,9,12,13,18,19,22 \}$ & $\lceil 4n/11 \rceil = 9$ \\
35 & $\{ 0, 1, 6, 9, 10, 15, 16, 19, 24, 25, 26, 30, 34 \}$ & $\lceil 4n/11 \rceil = 13$ \\
\end{tabular}
\end{center}
\caption{Identifying codes in $C_n(1,3)$ for certain lengths $n$ improving the general constructions} \label{TableIDSpecificCodes}
\end{table}

\medskip

\begin{figure}[t]
\centering
\includegraphics[width=350pt]{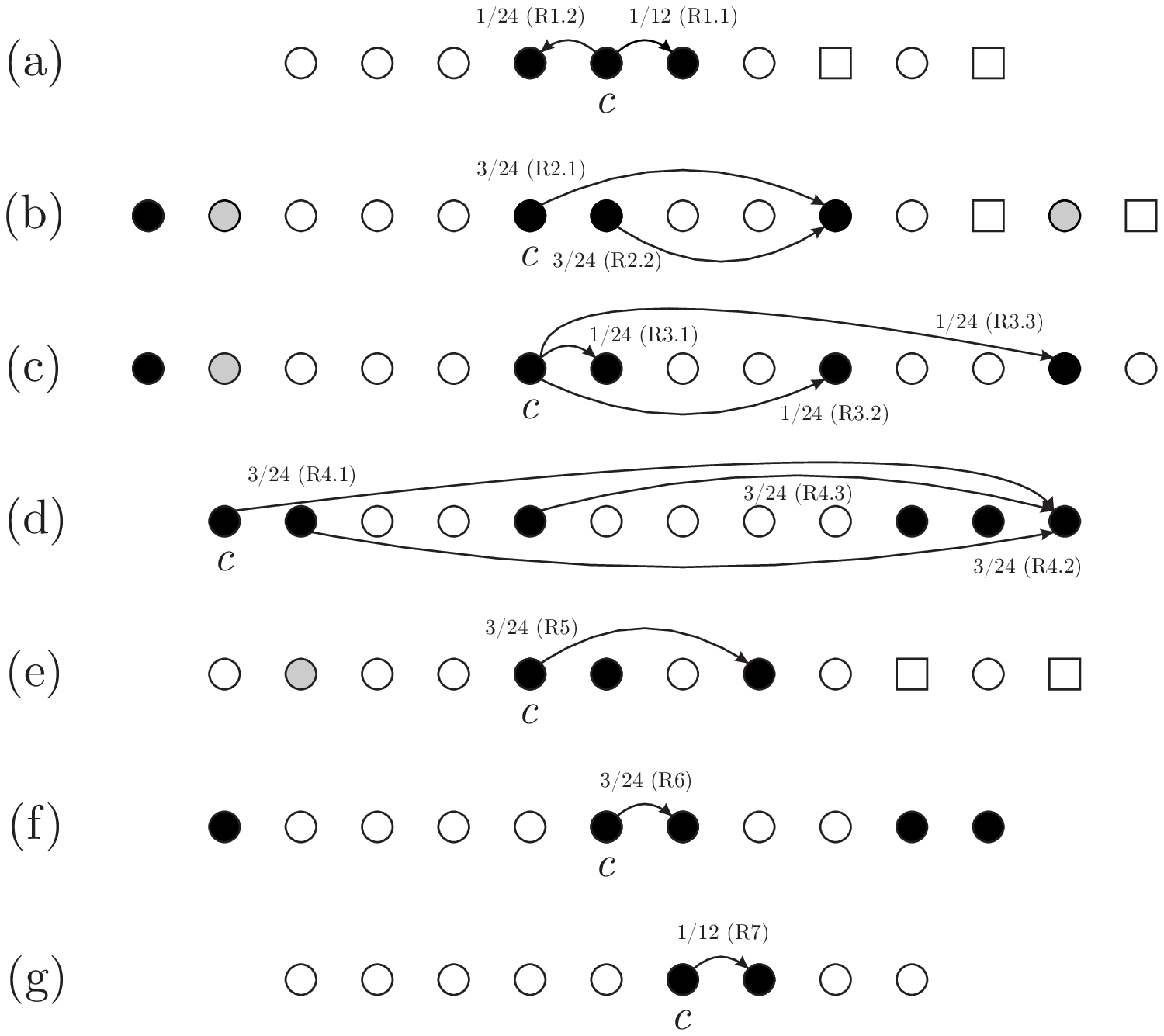}
\caption{The rules of the shifting scheme illustrated. The black dots represent codewords, the white dots represent non-codewords, and the grey dots can be either codewords or non-codewords. In the figures~(a), (b) and (e), at least one of the vertices marked with a white square is a codeword. Notice that the edges of the circulant graph are omitted in the figure.}
\label{IDRuleFigure}
\end{figure}
In what follows, we concentrate on improving the lower bound of $\ID(C_n(1,3))$ for $n \equiv 2, 5, 8 \pmod{11}$. For the rest of the section, assume first that $C$ is an identifying code in the circulant graph $C_n(1,3)$. For the lower bound on $|C|$, we introduce a shifting scheme to even out the share among the codewords as explained in Section~\ref{SectionLBShare}. The rules of the shifting scheme are illustrated in Figure~\ref{IDRuleFigure}. In addition to the rules shown in the figure, we also have rules which are obtained by reflecting the figures over the line passing vertically through the codeword $c$. For example, corresponding to Figure~\ref{IDRuleFigure}(a), we also have the symmetrical rules~R1.1' and R1.2'. In what follows, we describe more carefully how share is shifted by the rules:
\begin{itemize}
\item Let $c$ be a codeword such that its surroundings are as in Figure~\ref{IDRuleFigure}(a). In other words, $\{c-1, c, c+1\} \subseteq C$, $\{c-4, c-3, c-2, c+2, c+4\} \cap C = \emptyset$ and at least one of $c+3$ and $c+5$ is a codeword. Now $1/12$ units of share is shifted from $c$ to $c+1$ by the rule~R1.1 and $1/24$ units of share to $c-1$ by the rule~R1.2. Symmetrically, if $c$ is a codeword such that its surroundings are as Figure~\ref{IDRuleFigure}(a) when it is reflected over the line passing vertically through $c$, i.e., we have $\{c-1, c, c+1\} \subseteq C$, $\{c-2,c-4,c+2,c+3,c+4\} \cap C = \emptyset$ and at least one of $c-5$ and $c-3$ is a codeword, then $1/12$ units of share is shifted from $c$ to $c-1$ by the rule~R1.1' and $1/24$ units of share to $c+1$ by the rule~R1.2'.
\item If $c$ is a codeword such that its surroundings are as in Figure~\ref{IDRuleFigure}(b), then $3/24$ units of share is shifted to $c+4$ from $c$ by the rule~R2.1 and from $c+1$ by the rule~R2.2. In the symmetrical case, we have the analogous rules~R2.1' and R2.2'.
\item If $c$ is a codeword such that its surroundings are as in Figure~\ref{IDRuleFigure}(c), then $1/24$ units of share is shifted from $c$ to $c+1$ by the rule~R3.1, to $c+4$ by the rule~R3.2 and to $c+7$ by the rule~R3.3. In the symmetrical case, we have the analogous rules~R3.1', R3.2' and R3.3'.
\item If $c$ is a codeword such that its surroundings are as in Figure~\ref{IDRuleFigure}(d), then $3/24$ units of share is shifted to $c+11$ from $c$ by the rule~4.1, from $c+1$ by the rule~R4.2 and from $c+4$ by the rule~R4.3. In the symmetrical case, we have the analogous rules~R4.1', R4.2' and R4.3'.
\item If $c$ is a codeword such that its surroundings are as in Figure~\ref{IDRuleFigure}(e), then $3/24$ units of share is shifted from $c$ to $c+3$ by the rule~R5. In the symmetrical case, we have the analogous rule~R5'.
\item If $c$ is a codeword such that its surroundings are as in Figure~\ref{IDRuleFigure}(f), then $3/24$ units of share is shifted from $c$ to $c+1$ by the rule~R6. In the symmetrical case, we have the analogous rule~R6'.
\item If $c$ is a codeword such that its surroundings are as in Figure~\ref{IDRuleFigure}(g), then $1/12$ units of share is shifted from $c$ to $c+1$ by the rule~R7. In the symmetrical case, we have the analogous rule~R7'.
\end{itemize}
The modified share of a codeword $c \in C$, which is obtained after the shifting scheme has been applied, is denoted by $\ms(c)$. The usage of the shifting scheme is illustrated in the following example.
\begin{example}
Consider the identifying code $C_q$ in $C_{11q}(1,3)$. Observe first that we have $s(11i) = 1 + 3 \cdot 1/2 + 1/3 = 17/6 = 11/4 + 1/12$, $s(11i+1) = 1 + 2 \cdot 1/2 + 2 \cdot 1/3 = 8/3 = 11/4 - 1/12$, $s(11i+4) = 1 + 2 \cdot 1/2 + 2 \cdot 1/3 = 11/4 - 1/12$ and $s(11i+5) = 1 + 3 \cdot 1/2 + 1/3 = 11/4 + 1/12$ when $0 \leq i \leq q-1$. By the shifting scheme, $1/12$ units of share is shifted from each codeword $11i$ with $0 \leq i \leq q-1$ to $11i+1$ according to the rule~R7 and $1/12$ units of share is shifted from $11i+5$ to $11i+4$ according to the rule~R7'. Thus, after shifting scheme has been applied we have $\ms(c) = 11/4$ for all $c \in C_q$.
\end{example}

Let $u$ be a vertex in $C_n(1,3)$. We say that the consecutive vertices $u, u+1, \ldots, u+8$ form a \emph{pattern} $P$ (resp. $P'$) if $\{u+2, u+3\} \subseteq C$ and $\{u,u+1,u+4,u+5,u+6,u+7,u+8\} \cap C = \emptyset$ (resp. $\{u+5, u+6\} \subseteq C$ and $\{u,u+1,u+2,u+3,u+4,u+7,u+8\} \cap C = \emptyset$). Furthermore, we say that a codeword $c \in \Z_n$ belongs to a pattern $P$ (resp. $P'$) if $c$ is one of the codewords $u+2$ or $u+3$ (resp. $u+5$ or $u+6$) for some pattern $P$ (resp. $P'$). Observe that all the codewords in the identifying code $C_q$ belong to some pattern $P$ or $P'$. In what follows, we first show that after the shifting scheme has been applied the averaged share $\ms(c) \leq 65/24 = 11/4 - 1/24$ for any $c \in C$ unless the codeword $c$ belongs to some pattern $P$ or $P'$ when we have $\ms(c) \leq 11/4$. Recall that in~\cite{GNlidcn} Ghebleh and Niepel have shown using similar (albeit simpler) methods that on average the share of a codeword is at most $11/4$. Their method is based on a close study of connected components of codewords. Our refinement of the upper bound, which is based on recognizing the codewords achieving the upper bound of $11/4$ units of share, is essential to improving the lower bound for the lengths $n \equiv 2,5,8 \pmod{11}$ (as is shown later).

In what follows, we present two auxiliary lemmas for obtaining an upper bound on $\ms(u)$; in the first one, we consider codewords receiving share according to some rule and, in the second one, we study codewords not receiving any share. In the following lemma, we begin by presenting an upper bound on $\ms(u)$ when $u$ is a codeword receiving share according to some rule.
\begin{lemma} \label{LemmaIDReceivesShare}
Let $C$ be an identifying code in $C_n(1,3)$ and $u \in C$ be a codeword such that $u$ receives share according to the previous rules. If $u$ belongs to some pattern $P$ or $P'$, then we have $\ms(u) \leq 11/4$, and otherwise $\ms(u) \leq 65/24 = 11/4 - 1/24$.
\end{lemma}
\begin{proof}
Let $C$ be an identifying code in $C_n(1,3)$ and $u \in C$ be a codeword such that $u$ receives share according to some rule. 
The proof now divides into different cases depending on which rule(s) are applied to $u$. 

Suppose first that share is shifted to $u$ according to the rule~R1.1. Observe first that $|I(u+1)|\geq 3$ and $|I(u+3)| \geq 2$ since $u+2$ or $u+4$ belongs to $C$. Therefore, we have $s(u) \leq 3 \cdot 1/2 + 2 \cdot 1/3 \leq 13/6 = 11/4 - 7/12$. Furthermore, since $\{u-2,u-1\} \subseteq C$ and at least one of $u+2$ and $u+4$ is a codeword, we obtain that in addition to the rule~R1.1, $u$ can receive share only according to the rules~R1.2', R4.1, R4.2 and R4.3. Therefore, $\ms(u) \leq s(u) + 1/12 + 1/24 + 3 \cdot 3/24 \leq 8/3 = 11/4 - 2/24$ and we are done. If $u$ receives share according to the symmetrical rule~R1.1', then we are again done since the reasoning is analogous to the considered case.

Suppose that $u$ receives share according to the rule~R1.2; the case with the symmetrical rule~R1.2' is analogous. Now, as $u-1 \notin C$, $\{u+1,u+2\} \subseteq C$, and at least one of $u+4$ and $u+6$ is a codeword, it is straightforward to check that (in addition to~R1.2) $u$ can receive share only according to the rule~R1.1'. However, the case where $u$ receives share according to the rule~R1.1' has already been considered above. Hence, we may assume that share is received only according to the rule~R1.2. Thus, as $|I(u+3)|\geq 3$, we obtain that $s(u) \leq 1+2\cdot 1/2 + 2 \cdot 1/3 = 8/3 = 11/4 - 1/12$. Therefore, $\ms(u)  \leq s(u) + 1/24 \leq 11/4 - 1/24$ and we are done.

Suppose that $u$ receives share according to the rules~R2.1 and R2.2 (the case with the rules~R2.1' and R2.2' is analogous). Observe that since $u+2$ or $u+4$ is a codeword, at least one of the vertices $u+1$ and $u+3$ is adjacent to $3$ codewords as otherwise $I(u+1) = I(u+3)$. Therefore, we obtain that $s(u) \leq 3 \cdot 1/2 + 2 \cdot 1/3 = 13/6 = 11/4 - 7/12$. Furthermore, comparing the surroundings of $u$ to the ones in other rules, it can be deduced that (besides the rules~R2.1 and R2.2) $u$ can receive share only according to the rules~R2.1' and R2.2'. This implies that $\ms(u) \leq s(u) + 4 \cdot 3/24 =  11/4 - 2/24$ and we are done.

Suppose that $u$ receives share according to the rule~R3.1 (the case with R3.1' is analogous). Now it straightforward to check that $u$ cannot receive share according to any other rule. Furthermore, we have $s(u) = 1 + 2 \cdot 1/2 + 2 \cdot 1/3 = 8/3 = 11/4 - 2/24$. Therefore, we have $\ms(u) \leq s(u) + 1/24 = 11/4 - 1/24$ and we are done.

Suppose that $u$ receives share according to the rule~R3.2 (the case with R3.2' is analogous). Now we have $s(u) \leq 1 + 2 \cdot 1/2 + 2 \cdot 1/3 = 11/4 - 2/24$ as $|I(u-3)| = |I(u)| = 3$. Therefore, if share is not shifted to $u$ by any other rule, then we are immediately done since $\ms(u) \leq s(u) + 1/24 \leq 11/4 - 1/24$. Furthermore, it is straightforward to verify that in addition $u$ can only receive share according to the rule~R3.3'. Then $u-4$, $u-3$, $u+3$, $u+6$ and $u+7$ are codewords and $s(u) \leq 1 + 1/2 + 3 \cdot 1/3 = 5/2 = 11/4 - 1/4$. Therefore, we have $\ms(u) \leq s(u) + 2 \cdot 1/24 \leq 11/4 - 4/24$ and we are done.

Suppose that $u$ receives share according to the rule~R3.3 (the case with R3.3' is analogous). Observe first that if $u+2$ and $u+4$ are both non-codewords, then a contradiction follows as $I(u-1) = I(u+1) = \{u\}$. Hence, we may assume that $u+2$ or $u+4$ is a codeword. Therefore, one of the $I$-sets $I(u+1)$ and $I(u+3)$ contains at least $3$ codewords. Thus, we have $s(u) \leq 1 + 2 \cdot 1/2 + 2 \cdot 1/3 \leq 11/4 - 2/24$. Observe that the rule~R3.2' is the only other rule according to which $u$ can receive; in particular, notice that share cannot be received by the rule~3.3' since $u+2$ or $u+4$ is a codeword. Furthermore, the case where share is received according to the rule~R3.2' has already been considered above.

Suppose that $u$ receives share according to the rules~R4.1, R4.2 and R4.3 (the case with R4.1', R4.2' and R4.3' is analogous). Observe first that $u+1$, $u+2$ or $u+4$ belongs to $C$ since $I(u-3) \neq I(u+1)$. 
This implies that $s(u) \leq 3 \cdot 1/2 + 2 \cdot 1/3 \leq 13/6 = 11/4 - 7/12$. Furthermore, if $u$ receives no share according to any other rule or receives share according to the rule~1.1, then we are immediately done as in the case of the rule~R1.1. The only other possibility for $u$ to receive share is according to the rules~R4.1', R4.2' and R4.3'. However, in this case, the vertices $u-2$, $u-1$, $u$, $u+1$ and $u+2$ are all codewords. This implies that $s(u) \leq 2 \cdot 1/2 + 1/3 + 2 \cdot 1/4 = 11/6$. Therefore, we have $\ms(u) \leq s(u) + 6 \cdot 3/24 \leq 49/20 = 11/4 - 1/6 = 11/4 - 4/24$ and we are done.

Suppose that $u$ receives share according to the rule~R5 (the case with R5' is analogous). Now $u$ cannot receive share according to any other rule. Furthermore, as $u+2$ or $u+4$ is a codeword, we obtain that $s(u) \leq 3 \cdot 1/2 + 2 \cdot 1/3 = 13/6 = 11/4 - 7/12$. Therefore, we are immediately done since $\ms(u) \leq s(u) + 3/24 \leq 11/4 - 11/24$.

Suppose that $u$ receives share according to the rule~R6 (the case with R6' is analogous). Now it straightforward to verify that $u$ does not receive share according to any other rule. Furthermore, as $|I(u)| = 3$ and $|I(u+3)| \geq 3$, we immediately obtain that $s(u) \leq 3 \cdot 1/2 + 2 \cdot 1/3 = 13/6 = 11/4 - 14/24$. Hence, we are immediately done since $\ms(u) \leq s(u) + 3/24 \leq 11/4 - 1/24$.

Suppose that $u$ receives share according to the rule~R7 (the case with R7' is analogous). Again $u$ cannot receive according to any other rule. Observe first that $u+3$ and $u+4$ are codewords since $I(u-1) \neq I(u)$ and $I(u-3) \neq I(u+1)$. Therefore, as $|I(u)| \geq 3$ and $|I(u+3)| \geq 3$, we immediately obtain that $s(u) \leq 1 + 2 \cdot 1/2 + 2 \cdot 1/3 = 8/3 = 11/4 - 2/24$. Thus, we have $\ms(u) \leq s(u) + 1/12 \leq 11/4$. However, now this is enough since $u$ belongs to a pattern $P'$. Thus, in conclusion, the claim follows.
\end{proof}

In the following lemma, we give an upper bound on $\ms(u)$ when $u$ is a codeword not receiving share according to any rule.
\begin{lemma} \label{LemmaIDShiftsShare}
Let $C$ be an identifying code in $C_n(1,3)$ and $u \in C$ be a codeword such that $u$ does not receive share according to any of the previous rules. If $u$ belongs to some pattern $P$ or $P'$, then we have $\ms(u) \leq 11/4$, and otherwise $\ms(u) \leq 65/24 = 11/4 - 1/24$.
\end{lemma}
\begin{proof}
Let $C$ be an identifying code in $C_n(1,3)$ and $u \in C$ be a codeword such that $u$ does not receive share according to the rules. Observe first that if $u+2$ is a codeword, then we are immediately done since at least two of the $I$-sets $I(u-1)$, $I(u+1)$ and $I(u+3)$ consists of at least three codewords implying $\ms(u) \leq s(u) \leq 1 + 2 \cdot 1/2 + 2 \cdot 1/3 =  8/3 = 11/4 - 2/24$. The same argument also applies for $u-2 \in C$. Hence, we may assume that $u-2$ and $u+2$ do not belong to $C$. Now the proof divides into the following cases depending on the number of codewords in $I(u)$:
\begin{itemize}
\item Suppose first that $|I(u)|=1$, i.e., $I(u) = \{u \}$. The previous observation taken into account, we now know that $u-3$, $u-2$, $u-1$, $u+1$, $u+2$ and $u+3$ are non-codewords. Therefore, as $I(u) \neq I(u-1)$ and $I(u) \neq I(u+1)$, we obtain that $u-4$ and $u+4$ belong to $C$. Furthermore, since $I(u-3) \neq I(u-1) = \{u-4,u\}$ and $I(u+3) \neq I(u+1) = \{u, u+4\}$, we have $|I(u-3)| \geq 3$ and $|I(u+3)| \geq 3$. Hence, we have $\ms(u) \leq s(u) \leq 1 + 2 \cdot 1/2 + 2 \cdot 1/3 = 11/4 - 2/24$ and we are done.

\item Suppose then that $|I(u)|=2$. Now we have a further split into the cases with $I(u) = \{u-3,u\}$ and $I(u) = \{u,u+1\}$ (the cases with $I(u) = \{u+3,u\}$ and $I(u) = \{u-1,u\}$ are analogous). Consider first the case with $I(u) = \{u-3,u\}$. If now $u+4 \in C$, then $|I(u-3)| \geq 3$ and $|I(u+3)| \geq 3$ since $I(u-3) \neq I(u)$ and $I(u+1) \neq I(u+3)$ and we are done as $\ms(u) \leq s(u) \leq 1 + 2 \cdot 1/2 + 2 \cdot 1/3  = 11/4 - 2/24$. Hence, we may assume that $u+4 \notin C$. Therefore, as $I(u-1) \neq I(u+1) = \{u\}$, we have $u-4 \in C$. Furthermore, since $I(u+2) \neq \emptyset$, $I(u+1) \neq I(u+3)$ and $I(u+2) \neq I(u+4)$, we obtain respectively that $u+5$, $u+6$ and $u+7$ belong to $C$. Now $3/24$ units of share is shifted from $u$ to $u+7$ according to the rule~R4.3. Thus, we have $\ms(u) \leq s(u) - 3/24 \leq 1 + 3 \cdot 1/2 + 1/3 - 3/24 =  11/4 - 1/24$.

    For the other case, suppose that $I(u) = \{u,u+1\}$. Observe first that $u+4$ belongs to $C$ since $I(u) \neq I(u+1)$. It suffices to assume that $u-4 \notin C$ since otherwise $\ms(u) \leq s(u) \leq 3 \cdot 1/2 + 2 \cdot 1/3 = 13/6 = 11/4 - 7/12$ and we are done. If now $u-5 \notin C$, then $u+5 \in C$ as $I(u+2) \neq I(u-2) = \{u\}$ and $1/12$ units of share is shifted from $u$ to $u+1$ according to the rule~R7. Therefore, $\ms(u) \leq s(u) - 1/12 \leq 1 + 3 \cdot 1/2 + 1/3 - 1/12 =  11/4$ and we are done since $u$ belongs to a pattern~$P'$. Hence, we may assume that $u-5$ is a codeword. If $u+5$ is a codeword, then $3/24$ units of share is shifted from $u$ to $u+1$ according to the rule~R6 and we are again done since $\ms(u) \leq s(u) - 3/24 \leq 1 + 3 \cdot 1/2 + 1/3 - 3/24 =  11/4-1/24$. Hence, we may assume that $u+5 \notin C$. If at least one of $u+6$ and $u+8$ is a codeword, then $3/24$ units of share is shifted from $u$ to $u+4$ according to the rule~R2.1 Thus, we have $\ms(u) \leq s(u) - 3/24 \leq 1 + 3 \cdot 1/2 + 1/3 - 3/24 =  11/4-1/24$ and we are done. Hence, we may assume that $u+6$ and $u+8$ do not belong to $C$. If $u+7 \in C$, then $1/24$ units of share is shifted from $u$ to $u+1$, $u+4$ and $u+7$ according to the rules~R3.1, R3.2 and R3.3, respectively. Therefore, we have $\ms(u) \leq s(u) - 3 \cdot 1/24 \leq 1 + 3 \cdot 1/2 + 1/3 - 3/24 =  11/4-1/24$ and we are done. Hence, we may assume that $u+7$ is a non-codeword. Thus, since $I(u+6) \neq \emptyset$, $I(u+5) \neq I(u+7)$ and $I(u+6) \neq I(u+8)$, we obtain respectively that $u+9$, $u+10$ and $u+11$ belong to $C$. Now $3/24$ units of share is shifted from $u$ to $u+11$ according to the rule~R4.1. Therefore, we are again done since $\ms(u) \leq 11/4 - 1/24$. This concludes the proof of the current case.

\item Suppose then that $|I(u)|=3$. Observe first that if for some $v \in N(u)$ we have $|I(v)| \geq 3$, then we are immediately done since $\ms(u) \leq s(u) \leq 1 + 2 \cdot 1/2 + 2 \cdot 1/3 = 11/4 - 2/24$. Now, for $|I(u)|=3$, we have the following essentially different cases (others are analogous): $I(u) = \{u-3,u,u+3\}$, $I(u) = \{u-1,u,u+3\}$, $I(u) = \{u,u+1,u+3\}$ and $I(u) = \{u-1,u,u+1\}$. For future considerations, recall that the vertices $u-2$ and $u+2$ do not belong to $C$. Consider first the case with $I(u) = \{u-3,u,u+3\}$. By the previous observation, we may assume that $u-4$ and $u+4$ do not belong to $C$. However, this implies a contradiction since $I(u-1) = I(u+1) = \{u\}$.

    Consider then the case with $I(u) = \{u-1,u,u+3\}$. By the previous observation, we may assume that $u-4$, $u+4$ and $u+6$ are non-codewords. Thus, since $I(u-3) \neq I(u+1)  = \{u\}$, $u-6$ is a codeword. If $u+5$ or $u+7$ is a codeword, then $3/24$ units of share is shifted from $u$ to $u+3$ by the rule~R2.1. Therefore, we are done as $\ms(u) \leq s(u) - 3/24 \leq 1 + 3 \cdot 1/2 + 1/3 - 3/24 =  11/4-1/24$. Hence, we may assume that $u+5$ and $u+7$ do not belong to $C$. Thus, since $I(u+5) \neq \emptyset$, $I(u+4) \neq I(u+6)$ and $I(u+5) \neq I(u+7)$, we obtain respectively that $u+8$, $u+9$ and $u+10$ belong to $C$. Now $3/24$ units of share is shifted from $u$ to $u+10$ according to the rule~R4.2. Therefore, we are again done since $\ms(u) \leq 11/4 - 1/24$.

    Suppose then that $I(u) = \{u,u+1,u+3\}$. By the previous observation, we may assume that $u+4$ and $u+6$ are non-codewords. If $u-4 \in C$, then we are immediately done since $\ms(u) \leq s(u) \leq 3 \cdot 1/2 + 2 \cdot 1/3 = 13/6 = 11/4 - 7/12$. Hence, we may assume that $u+4 \notin C$. Now $u+5$ or $u+7$ belongs to $C$ as otherwise $I(u+2) = I(u+4) = \{u+1,u+3\}$. Therefore, $3/24$ units of share is shifted from $u$ to $u+3$ according to the rule~R5. Thus, we are done as $\ms(u) \leq s(u) - 3/24 \leq 1 + 3 \cdot 1/2 + 1/3 - 3/24 =  11/4-1/24$.

    Finally, suppose that $I(u) = \{u-1,u,u+1\}$. Now at least one of $u-5$ and $u+5$ is a codeword since $I(u-2) \neq I(u+2)$. Without loss of generality, we may assume that $u+5 \in C$. Then $1/24$ and $1/12$ units of share is shifted from $u$ to $u-1$ and $u+1$ according to the rules~R1.2 and R1.1, respectively. Therefore, we are done since $\ms(u) \leq s(u) - 1/24- 1/12 \leq 1 + 3 \cdot 1/2 + 1/3 - 3/24 =  11/4-1/24$.

\item Suppose then that $|I(u)|=4$. The proof now divides into the following essentially different cases: $I(u) = \{u-1, u, u+1, u+3\}$ and $I(u) = \{u-3, u, u+1, u+3\}$. In the former case, we may first assume that $u-4$ and $u+4$ are non-codewords by a similar argument as in the previous case. Then $1/24$ and $1/12$ units of share is shifted from $u$ to $u-1$ and $u+1$ according to the rules~R1.2 and R1.1, respectively. Therefore, we are done since $\ms(u) \leq s(u) - 1/24- 1/12 \leq 1 + 3 \cdot 1/2 + 1/4 - 3/24 =  11/4-3/24$.

    Suppose now that $I(u) = \{u-3, u, u+1, u+3\}$. By the previous observations, we may assume that $u-4$, $u+4$ and $u+6$ are non-codewords. Now $u+5$ or $u+7$ belongs to $C$ since $I(u+2) \neq I(u+4)$. Therefore, $3/24$ units of share is shifted from $u$ to $u+3$ according to the rule~R5. Thus, we are done as $\ms(u) \leq s(u) - 3/24 \leq 1 + 3 \cdot 1/2 + 1/4 - 3/24 =  11/4-3/24$.

\item Finally, suppose that $|I(u)| = 5$, i.e., $I(u) = \{u-3, u-1, u, u+1, u+3\}$. Now we are immediately done since $\ms(u) \leq s(u) \leq 1+ 3 \cdot 1/2 + 1/5 = 27/10 = 11/4 - 1/20 \leq 11/4 - 1/24$. This concludes the proof of the claim.
\end{itemize}
\end{proof}

In conclusion, the previous lemmas state that any codeword $c$ not belonging to a pattern~$P$ or $P'$ has $\ms(c) \leq 11/4 - 1/24$. In the following lemma, we consider the case where $C$ is an identifying code such that no codeword belongs to one of the patterns.
\begin{lemma} \label{LemmaNoPatterns}
Let $C$ be an identifying code in $C_n(1,3)$ such that no codeword of $C$ belongs to a pattern~$P$ or $P'$. Then the following results hold:
\begin{itemize}
\item If $n=11q_1+2$ with $q_1 \geq 5$, then $|C| \geq 4q_1+2 = \lceil 4n/11 \rceil +1$.
\item If $n=11q_2+5$ with $q_2 \geq 3$, then $|C| \geq 4q_2+3 = \lceil 4n/11 \rceil +1$.
\item If $n=11q_3+8$ with $q_3 \geq 1$, then $|C| \geq 4q_3+4 = \lceil 4n/11 \rceil +1$.
\end{itemize}
\end{lemma}
\begin{proof}
Let $C$ be an identifying code in $C_n(1,3)$ such that no codeword of $C$ belongs to a pattern~$P$ or $P'$. Denote $n = 11q + r$, where $q$ is a nonnegative integer and $r$ is an integer such that $0 \leq r < 11$. By Lemmas~\ref{LemmaIDReceivesShare} and \ref{LemmaIDShiftsShare}, we know that $\ms(c) \leq 65/24$ for all $c \in C$. Therefore, we obtain that
\[
n = \sum_{c \in C} s(c) = \sum_{c \in C} \ms(c) \leq \frac{65}{24}|C| \text{.}
\]
This further implies that
\[
|C| \geq \frac{24}{65}n = \frac{24}{65}\left(11q+r\right) = 4q + \frac{4q+24r}{65} \text{.}
\]
The rest of the proof now divides into the following cases:
\begin{itemize}
\item If $n=11q_1+2$ with $q_1 \geq 5$, then $|C| \geq 4q_1 + (4q_1+24 \cdot 2)/65 \geq 4q_1 + 68/65$. Therefore, we have $|C| \geq \lceil 4q_1 + 68/65 \rceil = 4q_1 + 2 = \lceil 4n/11 \rceil +1$.
\item If $n=11q_2+5$ with $q_2 \geq 3$, then $|C| \geq 4q_2 + (4q_2+24 \cdot 5)/65 \geq 4q_2 + 132/65$. Therefore, we have $|C| \geq \lceil 4q_2 + 132/65 \rceil = 4q_2 + 3 = \lceil 4n/11 \rceil +1$.
\item If $n=11q_3+8$ with $q_3 \geq 1$, then $|C| \geq 4q_3 + (4q_3+24 \cdot 8)/65 \geq 4q_3 + 196/65$. Therefore, we have $|C| \geq \lceil 4q_3 + 196/65 \rceil = 4q_3 + 4 = \lceil 4n/11 \rceil +1$.
\end{itemize}
\end{proof}

In the following theorem, we improve the lower bound on $\LD(C_n(1,3))$ for lengths $n$ such that $n$ is large enough and $n \equiv 2, 5, 8 \pmod{11}$.
\begin{theorem} \label{TheoremIDLBs}
Let $n$ be a positive integer such that $n = 11q_1 + 2$ with $q_1 \geq 5$, $n = 11q_2 + 5$ with $q_2 \geq 3$, or $n = 11q_3 + 8$ with $q_3 \geq 1$. Now we have
\[
\ID(C_n(1,3)) \geq \left\lceil \frac{4n}{11} \right\rceil +1 \text{.}
\]
\end{theorem}
\begin{proof}
Let $C$ be an identifying code in $C_n(1,3)$. Recall that if no codeword of $C$ belongs to a pattern~$P$ or $P'$, then the claim immediately follows by Lemma~\ref{LemmaNoPatterns}. Hence, we may assume that there exist codewords of $C$ belonging to a pattern~$P$ or $P'$. Observe that if $u+2$ and $u+3$ are codewords belonging to a pattern $P$, then $u-2$, $u-1$, $u+9$ and $u+10$ belong to $C$ since $I(u+1) \neq I(u+5) = \{u+2\}$, $I(u+2) \neq I(u+3) = \{u+2, u+3\}$, $I(u+6) \neq I(u+4) = \{u+3\}$ and $I(u+7) \neq \emptyset$, respectively. Analogously, it can be shown that if $u+5$ and $u+6$ are codewords belonging to a pattern $P'$, then $u-2$, $u-1$, $u+9$ and $u+10$. Suppose first that all the codewords belong to a pattern~$P$ or $P'$. The previous observation implies that the code is formed by consecutive repetitions of $P$ and $P'$. (Indeed, if $u+2$ and $u+3$ are codewords belonging to a pattern $P$, then the codewords $u-2$ and $u-1$ as well as $u+9$ and $u+10$ belong to patterns $P'$.) Observe that consecutive patterns~$P$ and $P'$ form a segment of length $11$ (with $4$ codewords) similar to the identifying code $C_q$ given in Theorem~\ref{TheoremIDConstructions}. However, as now $n$ is not divisible by $11$, the identifying code $C$ cannot entirely be formed by the segments of length $11$.  Thus, we obtain that all the codewords cannot belong to a pattern~$P$ or $P'$. In other words, after a (finite) repetition of patterns~$P$ and $P'$, a codeword not belonging to the patterns has to appear. In what follows, we first show that the end of the repetition of the patterns~$P$ and $P'$ implies a drop of strictly more than $3/4$ units of share in the sum $\sum_{c \in C} \ms(c)$ compared to the average share of $11/4$, i.e., $\sum_{c \in C} \ms(c) < \tfrac{11}{4} |C| - \tfrac{3}{4}$. Based on this observation, we then show that the original lower bound of $\lceil 4n/11 \rceil$ can be improved by one.

Suppose first that the repetition of the patterns ends with a pattern~$P$. More precisely, let $u-7$ and $u-6$ be codewords belonging to a pattern~$P$, and assume that the next codeword to the right does not belong to a pattern~$P'$. Recall that due to the pattern $P$ the vertices $u-9$, $u-8$, $u-5$, $u-4$, $u-3$, $u-2$ and $u-1$ are non-codewords. Now $u$ and $u+1$ belong to $C$ since $I(u-3) \neq I(u-5) = \{u-6\}$ and $I(u-2) \neq \emptyset$, respectively. By the assumption that $u$ (and $u+1$) do no belong to a pattern~$P'$, we can deduce that $u+2$ or $u+3$ is a codeword of $C$. These two cases are considered in the following:
\begin{itemize}
\item[(A1)] Suppose first that $u+2 \in C$. If $u+3 \in C$, then no share is shifted to $u$ according to any rule and we have $\ms(u) \leq s(u) \leq 2 \cdot 1/2 + 3\cdot 1/3 =  2 = 11/4 -3/4$. Furthermore, by Lemmas~\ref{LemmaIDReceivesShare} and \ref{LemmaIDShiftsShare}, we have $\ms(u+1) \leq 11/4-1/24$. Therefore, we are done since $\ms(u) + \ms(u+1) < 2 \cdot 11/4 - 3/4$. Hence, we may assume that $u+3$ does not belong to $C$. Now at least one of $u+4$ and $u+6$ is a codeword since $I(u-1) \neq I(u+3)$. Now we have $s(u) \leq 3 \cdot 1/2 + 2 \cdot 1/3 = 11/4 - 7/12$ and similarly $s(u+2) \leq 11/4 - 7/12$. Moreover, it is straightforward to verify that $u$ and $u+2$ can receive share only according to the rules~R1.2 and R1.1, respectively. Therefore, we obtain that $\ms(u) + \ms(u+2) \leq (s(u) + 1/24) + (s(u+2) + 1/12) \leq 2 \cdot 11/4 - 25/24 < 2 \cdot 11/4 - 3/4$. This concludes the first case of the proof.

\item[(A2)] Suppose then that $u+2 \notin C$ and $u+3 \in C$. Observe first that $u+4$ or $u+7$ is a codeword since otherwise $I(u+2) = I(u+4) = \{u+1, u+3\}$ (a contradiction). Suppose first that $u+4$ is a codeword. Observe then that $|I(v)| \geq 3$ for all $v \in \{u, u+1, u+3, u+4\}$. Therefore, we have $s(v) \leq 1 + 1/2 + 3 \cdot 1/3$ for all $v \in \{u, u+1, u+3, u+4\}$. It is straightforward to verify that $u$ and $u+1$ do not receive share according to any rule. Moreover, either $u+3$ or $u+4$ can receive share according to the rules~R4.1', R4.2' and R4.3'. Furthermore, if this happens for one of the vertices, say $v$, then we have $\ms(v) \leq 11/4 - 1/24$ by the previous lemmas and the other one does not receive share according to the rules. Thus, all the previous combined, we are done since $\ms(u) + \ms(u+1) + \ms(u+3) + \ms(u+4) \leq 4 \cdot 11/4 - 3 \cdot 1/4 - 1/24 < 4 \cdot 11/4 - 3 \cdot 1/4$. Hence, we may assume that $u+4 \notin C$ and $u+7 \in C$.

    Suppose that $u+6$ is a codeword. Then it is straightforward to verify that $u+3$ can receive share only according to the rules~R2.1' and R2.2'. Therefore, since $s(u+3) \leq 1/2 + 4 \cdot 1/3 = 11/6$, we obtain that $\ms(u+3) \leq s(u+3) + 2 \cdot 3/24 \leq 11/4 - 16/24$. Thus, we are done as $\ms(u) + \ms(u+1) + \ms(u+3) + \ms(u+6) \leq 3(11/4 - 1/24) + 11/4 - 16/24 = 4 \cdot 11/4 -19/24 < 4 \cdot 11/4 - 3/4$. Hence, we may assume that $u+6 \notin C$. Suppose then that $u+5$ or $u+9$ is a codeword; denote the codeword by $v$. Now we have $s(u+3) \leq 2 \cdot 1/2 + 3 \cdot 1/3$. Moreover, $u+3$ can receive only $3/24$ units of share according to the rule~R5. Therefore, we obtain that $\ms(u) + \ms(u+1) + \ms(u+3) + \ms(u+7) + \ms(v) \leq 4(11/4 - 1/24) + 11/4 - 15/24 = 4 \cdot 11/4 - 19/24 < 4 \cdot 11/4 - 3/4$. Hence, we may assume that $u+5$ and $u+9$ are both non-codewords. Now $u+8$ is a codeword since $I(u+5) \neq \emptyset$. Furthermore, at least one of $u+10$ and $u+12$ is a codeword, say $v$, since $I(u+5) \neq I(u+9)$. Now we have $s(u+3) \leq 3 \cdot 1/2 + 2 \cdot 1/3 = 11/4 - 14/24$ and as above $u+3$ can receive only $3/34$ units of share according to the rule~R5. Moreover, we have $s(w) \leq 1 + 2 \cdot 1/2 + 2 \cdot 1/3 = 11/4 - 2/24$ for any $w \in \{u+1, u+7, u+8\}$ and none of the codewords receive share according to any rule. Furthermore, we have $\ms(u) \leq 11/4 - 1/24$ and $\ms(v) \leq 11/4 - 1/24$ by Lemmas~\ref{LemmaIDReceivesShare} and \ref{LemmaIDShiftsShare} since neither of the vertices $u$ and $v$ belongs to a pattern~$P$ or $P'$. Thus, combining the previous observation, we obtain that $\ms(u) + \ms(u+1) + \ms(u+3) + \ms(u+7) + \ms(u+8) + \ms(v) \leq 2(11/4 - 1/24) + 3(11/4 - 2/24) + (11/4 - 14/24 + 3/24) = 6 \cdot 11/4 - 19/24 < 6 \cdot 11/4 - 3/4$. This concludes the proof of the current case.
\end{itemize}

Suppose then that the repetition of the patterns ends with a pattern~$P'$. More precisely, let $u-4$ and $u-3$ be codewords belonging to a pattern~$P'$, and assume that the next codeword to the right does not belong to a pattern~$P$. Recall that due to the pattern $P'$ the vertices $u-9$, $u-8$, $u-7$, $u-6$, $u-5$, $u-2$ and $u-1$ are non-codewords. Now $u$ and $u+1$ belong to $C$ since $I(u-3) \neq I(u-4) = \{u-4,u-3\}$ and $I(u-2) \neq I(u-6) = \{u-3\}$, respectively. By the assumption that $u$ (and $u+1$) do no belong to a pattern~$P$, we can deduce that one of the vertices $u+2$, $u+3$, $u+4$, $u+5$ and $u+6$ is a codeword of $C$. The proof now divides into the following five cases:
\begin{itemize}
\item[(B1)] Suppose that $u+2 \in C$. Now we have $s(u) \leq 1/2 + 4 \cdot 1/3 = 11/4 - 22/24$.  Furthermore, $u$ can receive share only according to the rules~R4.1', R4.2' and R4.3'. Obviously, if $u$ receives no share, then we are immediately done as $\ms(u) \leq s(u) \leq 11/4 - 22/24 \leq 11/4 - 3/4$. Hence, we may assume that share is shifted to $u$ according to the rules~R4.1', R4.2' and R4.3'. This implies that $u+7$, $u+10$ and $u+11$ are codewords. Therefore, we have $s(u+1) \leq 3 \cdot 1/2 + 2 \cdot 1/3 = 11/4 - 14/24$ and $u+1$ cannot receive share according to any rule. Thus, we are done since $\ms(u) + \ms(u+1) \leq (11/4 - 22/24 + 3 \cdot 3/24) + (11/4 - 14/24) = 2 \cdot 11/4 - 27/24 \leq 2 \cdot 11/4 - 3/4$.

\item[(B2)] Suppose that $u+2 \notin C$ and $u+3 \in C$. Now we have $s(u) \leq 3 \cdot 1/2 + 1/3 + 1/4 = 11/4 - 16/24$ and similarly $s(u+1) \leq 11/4 - 16/24$ (as $I(u+2) \neq I(u+4)$).  Hence, as neither $u$ nor $u+1$ receives share according to any rule, we obtain that $\ms(u) + \ms(u+1) \leq s(u) + s(u+1) \leq 2 (11/4 - 16/24) < 2 \cdot 11/4 - 3/4$. Thus, we are done.

\item[(B3)] Suppose that $u+2, u+3 \notin C$ and $u+4 \in C$. Now we have $s(u) \leq 2 \cdot 1/2 + 3 \cdot 1/3 = 2 = 11/4 - 3/4$. Furthermore, $u$ does not receive share according to any rule. Therefore, as $u+1$ does not belong to any pattern $P$ or $P'$, we are done since $\ms(u) + \ms(u+1) \leq (11/4 - 3/4) + (11/4 - 1/24) = 2 \cdot 11/4 - 19/24 < 2 \cdot 11/4 - 3/4$.

\item[(B4)] Suppose that $u+2, u+3, u+4 \notin C$ and $u+5 \in C$. Observe first that $u+7 \in C$ since $I(u+4) \neq I(u+2) = \{u+1, u+5\}$. Now we have $s(u) \leq 1 + 2 \cdot 1/2 + 2 \cdot 1/3 = 11/4 - 2/24$ and $s(u+1) \leq 3 \cdot 1/2 + 2 \cdot 1/3 = 11/4 - 14/24$. Furthermore, neither $u$ nor $u+1$ receives share according to any rule. Moreover, at least one of $u+6$, $u+8$ and $u+11$, say $v$, is a codeword since $I(u+6) \neq I(u+8)$. Observe that if $v = u+6$ or $v = u+8$, then $v$ does not belong to any pattern~$P$ or $P'$. Assuming $u+6$ and $u+8$ do not belong to $C$, then $v = u+11$ does not belong to $P$ or $P'$. Therefore, we are done as  $\ms(u) + \ms(u+1) + \ms(u+5) + \ms(u+7) + \ms(v) \leq (11/4 - 2/24) + (11/4 - 14/24) + 3(11/4 - 1/24) = 5 \cdot 11/4 - 19/24 < 5 \cdot 11/4 - 3/4$.

\item[(B5)] Finally, suppose that $u+2, u+3, u+4, u+5 \notin C$ and $u+6 \in C$. Observe first that $u+7 \in C$ since $I(u+4) \neq I(u+2) = \{u+1\}$. This implies that $s(u) \leq 3 \cdot 1/2 + 2 \cdot 1/3 = 11/4 - 14/24$. It is also straightforward to verify that $u$ can only receive $3/24$ units of share according to the rule~R6'. Therefore, we have $\ms(u) \leq s(u) + 3/24 \leq 11/4 - 11/24$. Furthermore, since $I(u+6) \neq I(u+7)$, we know that at least one of $u+8$, $u+9$ and $u+10$ has to be a codeword. Suppose first that $u+8 \in C$. Now we have $s(u+6) \leq 3 \cdot 1/2 + 2 \cdot 1/3 = 11/4 - 14/24$ (as $I(u+5) \neq I(u+9)$), and $u+6$ can only receive $1/24$ units of share according to the rule~R1.2. (In particular, notice that if share is shifted to $u+6$ according to the rules~R4.1', R4.2' or R4.3', then $I(u+5) = I(u+9)$ implying a contradiction.) Thus, we have $\ms(u) + \ms(u+6) \leq (11/4 - 11/24) + (11/4 - 14/24 + 1/24) \leq 2 \cdot 11/4 - 1 < 2 \cdot 11/4 - 3/4$. Hence, we may assume that $u+8 \notin C$. Suppose then that $u+9 \in C$. Now we have $s(u+7) \leq 3 \cdot 1/2 + 2 \cdot 1/3 = 11/4 - 14/24$, and $u+7$ cannot receive share according to any rule. Therefore, we are done as $\ms(u) + \ms(u+7) \leq (11/4 - 11/24) + (11/4 - 14/24) = 2 \cdot 11/4 - 25/24 < 2 \cdot 11/4 - 3/4$. Hence, we may assume that $u+9 \notin C$ and $u+10 \in C$.

    Suppose first that $u+11 \in C$. Now we have $s(u+7) \leq 3 \cdot 1/2 + 2 \cdot 1/3 = 11/4 - 14/24$, and $u+7$ can receive share only according to the rule~R6 ($3/24$ units). Therefore, we are done since $\ms(u) + \ms(u+7) \leq (11/4 - 11/24) + (11/4 - 14/24 + 3/24) = 2 \cdot 11/4 - 22/24 < 2 \cdot 11/4 - 3/4$. Hence, we may assume that $u+11 \notin C$. Suppose then that $u+12 \in C$ or $u+14 \in C$. This implies that $s(u+10) \leq 3 \cdot 1/2 + 2 \cdot 1/3 = 11/4 - 14/24$. Furthermore, $u+10$ receives share according to the rules~R2.1 ($3/24$ units) and R2.2 ($3/24$ units), and it can possibly receive share also by the rules~R2.1' ($3/24$ units) and R2.2' ($3/24$ units). If no share is shifted to $u+10$ according to the rules~R2.1' and R2.2', then we are done since $\ms(u) + \ms(u+10) \leq (11/4 - 11/24) + (11/4 - 14/24 + 2 \cdot 3/24) = 2 \cdot 11/4 - 19/24 < 2 \cdot 11/4 - 3/4$. Hence, we may assume that $u+10$ receives share also according to the rules~R2.1' and R2.2'. This implies that $u+13$, $u+14$ and $u+19$ are codewords of $C$. Observe that the codewords $u+1$, $u+6$, $u+10$, $u+13$, $u+14$ and $u+19$ do not belong to any pattern~$P$ or $P'$. In particular, $u+19$ does not belong to $P$ or $P'$ since $u+15 \notin C$. Thus, all the previous taken into account, we obtain that $\ms(u) + \ms(u+1) + \ms(u+6) + \ms(u+7) + \ms(u+10) + \ms(u+13) + \ms(u+14) + \ms(u+19) \leq (11/4 - 11/24) + (11/4 - 14/24 + 4 \cdot 3/24) + 6(11/4 - 1/24) = 8 \cdot 11/4 - 19/24 < 8\cdot 11/4 - 3/4$. Hence, we may assume that $u+12 \notin C$ and $u+14 \notin C$.

    Suppose that $u+13 \notin C$. Now $u+15$, $u+16$ and $u+17$ belong to $C$ since $I(u+12) \neq \emptyset$, $I(u+13) \neq I(u+11) = \{u+10\}$ and $I(u+14) \neq I(u+12) = \{u+15\}$, respectively. Furthermore, at least one of the vertices $u+18$, $u+19$ and $u+21$, say $v$, is a codeword since $I(u+14) \neq I(u+18)$. Thus, if $v=u+18$, $v=u+19$, or $v=u+21$ and $v$ does not belong to $P$ or $P'$, then we are done as $\ms(u) + \ms(u+1) + \ms(u+6) + \ms(u+7) + \ms(u+10) + \ms(u+15) + \ms(u+16) + \ms(u+17) + \ms(v) \leq (11/4 - 11/24) +  8(11/4 - 1/24) = 9 \cdot 11/4 - 19/24 < 9 \cdot 11/4 - 3/4$ (none of the other codewords either belong to a pattern~$P$ or $P'$). Hence, we may assume that $v$ belongs to a pattern~$P$ or $P'$. This implies that $u+18, u+19 \notin C$ and $v = u+21$. Now $u+20$ also belongs to the pattern~$P$ and the codewords $u+15$, $u+16$ and $u+17$ form a case symmetrical to the case~(B1). Hence, we may assume that $u+13 \in C$. Now $u+17 \in C$ because $I(u+14) \neq I(u+12)$. It is straightforward to verify that $u+13$ can now receive share only according to the rules~R3.3 ($1/24$ units) and R3.2' ($1/24$ units). If $u+15$ is a codeword, then $s(u+13) \leq 2 \cdot 1/2 + 3 \cdot 1/3 = 11/4 - 3/4$. Furthermore, if $u+16 \in C$, then we have $s(u+13) \leq 1 + 1/2 + 3 \cdot 1/3 = 11/4 - 6/24$. Thus, in both cases, we have $\ms(u+13) \leq s(u+13) + 2 \cdot 1/24 \leq 11/4 - 4/24$. Therefore, all the previous taken into account, we are done since $\ms(u) + \ms(u+1) + \ms(u+6) + \ms(u+7) + \ms(u+10) + \ms(u+13) + \ms(u+17) \leq (11/4 - 11/24) +  (11/4 - 4/24) + 5(11/4 - 1/24) = 7 \cdot 11/4 - 20/24 < 7 \cdot 11/4 - 3/4$ (none of the codewords belong to a pattern~$P$ or $P'$). Hence, we may assume that $u+15$ and $u+16$ are non-codewords. Now $u+18$ and $u+19$ belong to $C$ since $I(u+15) \neq \emptyset$ and $I(u+16) \neq I(u+14) = \{u+13, u+17\}$, respectively. Therefore, we have $\ms(u) + \ms(u+1) + \ms(u+6) + \ms(u+7) + \ms(u+10) + \ms(u+13) + \ms(u+17) + \ms(u+18) + \ms(u+19) \leq (11/4 - 11/24) + 8(11/4 - 1/24) = 9 \cdot 11/4 - 19/24 < 7 \cdot 11/4 - 3/4$ (again none of the codewords belong to a pattern~$P$ or $P'$). Thus, in conclusion, we achieve a drop of more than $3/4$ units of share in the sum $\sum_{c \in C} \ms(c)$ in all the cases compared to the average share of $11/4$, i.e., $\sum_{c \in C} \ms(c) < \tfrac{11}{4} |C| - \tfrac{3}{4}$.
\end{itemize}

In the previous detailed case analysis, we have achieved a drop of more than $3/4$ units of share in the sum $\sum_{c \in C} \ms(c)$. In what follows, we show how this implies the improved lower bound. The proof now splits into the following cases depending on the remainder when $n$ is divided by $11$:
\begin{itemize}
\item Suppose first that $n=11q_1+2$ with $q_1 \geq 5$. By the previous considerations, we now have
    \[
    n = \sum_{c \in C} s(c) = \sum_{c \in C} \ms(c) < \frac{11}{4}|C| - \frac{3}{4} \text{.}
    \]
    This implies that
    \[
    |C| > \frac{4}{11} \left( n + \frac{3}{4} \right) = 4q_1 + 1 \text{.}
    \]
    Thus, we have $|C| \geq 4q_1+2 = \lceil 4n/11 \rceil +1$.
\item Suppose then that $n = 11q_2 + 5$ with $q_2 \geq 3$. As in the previous case, we obtain that
    \[
    |C| > \frac{4}{11} \left( n + \frac{3}{4} \right) = 4q_2 + 2 + \frac{1}{11} \text{.}
    \]
    Thus, we have $|C| \geq 4q_2+3 = \lceil 4n/11 \rceil +1$.
\item Finally, suppose then that $n = 11q_3 + 8$ with $q_3 \geq 1$. As in the previous case, we obtain that
    \[
    |C| > \frac{4}{11} \left( n + \frac{3}{4} \right) = 4q_3 + 3 + \frac{2}{11} \text{.}
    \]
    Thus, we have $|C| \geq 4q_3+4 = \lceil 4n/11 \rceil +1$.
\end{itemize}
Thus, in conclusion, we have shown that $\ID(C_n(1,3)) \geq \lceil 4n/11 \rceil +1$ for $n = 11q_1 + 2$ with $q_1 \geq 5$, $n = 11q_2 + 5$ with $q_2 \geq 3$, and $n = 11q_3 + 8$ with $q_3 \geq 8$.
\end{proof}

Recall the general constructions of Theorem~\ref{TheoremIDConstructions}, the constructions for the specific lengths in Table~\ref{TableIDSpecificCodes} and the improved lower bound of Theorem~\ref{TheoremIDLBs}. Combining all these results, we know the exact values of $\ID(C_n(1,3))$ for all the lengths $n$ except for $n = 46$. In the open case $n = 46$, we have $17 = \lceil 4n/11 \rceil \leq \ID(C_{n}(1,3)) \leq \lceil 4n/11 \rceil + 1 = 18$ by the general lower and upper bounds. Using an exhaustive computer search, it can be shown that there does not exist an identifying code in $C_{46}(1,3)$ with 17 codewords, i.e., $\ID(C_{46}(1,3)) = 18$. The method of the exhaustive search is briefly explained in the following remark.
\begin{remark}
Let $C$ be a code in $C_{46}(1,3)$ with $17$ codewords. Without loss of generality, we may assume that $8$ of the codewords belong to $\{0,1, \ldots, 22\}$ and the rest $9$ codewords belong to $\{23, 24, \ldots, 45\}$. Observe that if $C$ is an identifying code in $C_{46}(1,3)$, then the vertices in $\{3,4, \ldots, 19\}$ have a unique identifying set among the codewords in $\{0,1, \ldots, 22\}$ and the vertices in $\{26, 27, \ldots, 42\}$ have a unique identifying set among the codewords in $\{23, 24, \ldots, 45\}$. Using a computer search, we obtain that there exist $1919$ codes $C_1 \subseteq \{0,1, \ldots, 22\}$ with $|C_1| = 8$ such that $I(C_1; u)$, where $u \in \{3,4, \ldots, 19\}$, are all non-empty and unique, and $23137$ codes $C_2 \subseteq \{23, 24, \ldots, 45\}$ with $|C_2| = 9$ such that $I(C_2; u)$, where $u \in \{26, 27, \ldots, 42\}$, are all non-empty and unique. By an exhaustive search, we obtain that no union of such codes $C_1$ and $C_2$ is an identifying code in $C_{46}(1,3)$. Therefore, by the previous observation, there does not exists an identifying code in $C_{46}(1,3)$ with 17 codewords. Hence, we have $\ID(C_{46}(1,3)) = 18$.
\end{remark}

The following theorem summarizes the results of the section and gives the exact values of $\ID(C_n(1,3))$ for all lengths $n \geq 11$. The exact values of $\ID(C_n(1,3))$ for the lengths $n$ smaller than $11$ have been determined in~\cite{GNlidcn}.
\begin{theorem} \label{TheoremIDConclusion}
Let $n$ be an integer such that $n \geq 11$. Now we have the following results:
\begin{itemize}
\item Assume that $n \leq 37$. If $n \equiv 8 \pmod{11}$, then we have $\ID(C_n(1,3)) = \lceil 4n/11 \rceil +1$, and otherwise $\ID(C_n(1,3)) = \lceil 4n/11 \rceil$.
\item Assume that $n \geq 38$. If $n \equiv 2, 5, 8 \pmod{11}$, then we have $\ID(C_n(1,3)) = \lceil 4n/11 \rceil +1$, and otherwise $\ID(C_n(1,3)) = \lceil 4n/11 \rceil$.
\end{itemize}
\end{theorem}


\section{Locating-dominating codes in $C_n(1,3)$} \label{SectionLD}

In this section, we consider locating-dominating codes in the circulant graph $C_n(1,3)$. For the rest of the section, assume that $C$ is a locating-dominating code in $C_n(1,3)$.  Recall that we have $\LD(C_n(1,3)) = \lceil n/3 \rceil$ if $n \equiv 0, 1, 4 \pmod{6}$ and $\lceil n/3 \rceil \leq \LD(C_n(1,3)) \leq \lceil n/3 \rceil +1$ if $n \equiv 2, 3, 5 \pmod{6}$ by~\cite{GNlidcn}. Moreover, it is conjectured that $\LD(C_n(1,3)) = \lceil n/3 \rceil +1$ if $n \equiv 2, 3, 5 \pmod{6}$ (see Conjecture~\ref{ConjectureLD}). In what follows, we prove the stated conjecture by increasing the lower bound on $\LD(C_n(1,3))$ for the lengths $n \equiv 2, 3, 5 \pmod{6}$. The basic idea of the proof is similar to the one in the case of identifying codes. In what follows, we show that the average share of a codeword is now at most $17/6$ unless the codeword belongs to a specific pattern of codewords and non-codewords (namely, the pattern~$S3$ which is defined later) when the average share is at most $3$. However, for making the proof more convenient and illustrative, the technical organization of the proof is somewhat different.

In the following proofs, we use several different patterns consisting of codewords and non-codewords. In the illustrations of the patterns, the letter~$x$ denotes a codeword, the letter~$o$ denotes a non-codeword and the symbol~$*$ means that the vertex can be either a codeword or a non-codeword. For example, the pattern~$S1$ is defined as follows: $x*ooox \underline{x} oooo*x$. We say that a pattern is in $C_n(1,3)$ if there exists such a consecutive segment of codewords and non-codewords in the graph. Moreover, if a pattern is in the graph, then we say that a codeword belongs to the pattern if it is the underlined codeword of the pattern. Observe that there exists a symmetrical version of each pattern. For example, the symmetrical version of the pattern~$S1$ is $x*oooo \underline{x} xooo*x$. As in the case of identifying codes, we could use a notation $S1'$ for the symmetrical version of $S1$. However, for simplicity, this notation is omitted and we denote both the original and the symmetric pattern by $S1$.

In the following proposition, we consider the share $s(c)$ of a codeword $c \in C$. In particular, we show that $s(c) \leq 17/6$ in most of the cases and determine the exact situations when $s(c) \geq 3$.
\begin{proposition}\label{shares}
Let $n\geq 14$ and $C$ be a locating-dominating code in $C_n(1,3)$. For all $c\in C$, we have  either $s(c)\le 17/6$ or $s(c)\in \{3,37/12,10/3\}.$
Moreover, the following statements hold:
\begin{itemize}
\item $s(c) = 3$ if and only if $c$ belongs to a pattern~$S1$ or $S3$ (defined below).
\item $s(c) = 37/12$ if and only if $c$ belongs to a pattern~$S4$ (defined below).
\item $s(c) = 10/3$ if and only if $c$ belongs to a pattern~$S6$ (defined below).
\end{itemize}
\end{proposition}
\begin{proof}
  Let $c$ be a codeword in $C$. The proof now divides into three parts depending on whether $|I(c)| \geq 3$, $|I(c)| \geq 2$ or $|I(c)| = 1$.
  \begin{itemize}
  \item Suppose first that $|I(c)|\ge 3$. Observe that there exists at most one vertex $u$ in $N[c]$ such that $|I(u)| = 1$, and the other vertices are covered by at least two codewords. Hence, we immediately obtain that $s(c) \leq 1 + 3 \cdot 1/2 + 1/3 = 17/6$.
  \item Assume then that $|I(c)|=2$. If all the vertices  $v\in N[c]$ have $|I(v)|\ge 2$, then we get $s(c)\le 5/2<17/6.$ Therefore, it is enough to consider the case where there is at least one vertex $v\in N[c]$ with $|I(v)|=1.$ There cannot be more than one such vertex. Indeed, such a vertex must be a non-codeword, and if there were two, say $u$ and $w$, then $I(u)=I(w)$, which is not possible. Moreover, if there is one vertex $v\in N[c]$ such that $|I(v)|\ge 3$, we have $s(c)\le 17/6.$ Therefore, $s(u)=3$ if and only if all the vertices in $N[c]$ have the size of the $I$-sets equal to 2 except one equal to 1. Next we analyze this case more carefully.
    \begin{itemize}
    \item Let first $c-3\in I(c)$ (the case $c+3$ goes analogously). If $I(c-1)=\{c\}$ (resp. $I(c+1)=\{c\}$), then $c+4$ (resp. $c-4$) belongs to $C$ implying $|I(c+3)|\ge 3$ (resp. $|I(c-3)|\ge 3$). If $I(c+3)=\{c\}$, then $|I(c-3)|\ge 3$ (since $I(c-1) \neq I(c+1) = \{c\}$). In all cases, the share is at most $17/6$.
    \item Assume then that $c-1\in I(c)$ (the case $c+1$ is analogous). If $I(c-3)=\{c\}$, then  $|I(c+3)|\ge 3$. If $I(c+3)=\{c\}$, then $|I(c-1)|\ge 3.$ In these cases $s(c)\le 17/6.$ Therefore, we can assume that $I(c+1)=\{c\}$ and $|I(c-3)|=|I(c+3)|=2$. Due to $c+3$, we must have $c+6\in C.$ Now either $c-4$ or $c-6$ belong to $C$ (if both we are done). Moreover, we may assume that $c-4 \notin C$ as otherwise $|I(c-1)| \geq 3$ implying $s(c)\le 17/6$. Therefore, it is enough to consider the case $c-6\in C.$ Consequently, we have the pattern:
      $$ x*ooox\underline{x}oooo*x \text{,}$$
      where $c$ is denoted by the underlined codeword $\underline{x}$. Both of the unknowns cannot be non-codewords because then $I(c-2)=I(c+2)$ and $c-2,c+2\notin C.$ Moreover, we have $c-7\in C$ since otherwise $I(c-2)=I(c-4).$ This leads to the following two patterns when $s(c)=3$:
           $$
      \begin{array}{lc}
        S1 & xxxooox\underline{x}oooo*x\\
        S3 & xxoooox\underline{x}ooooxx
        \end{array}.
      $$
      \end{itemize}
  \item Let then $|I(c)|=1$. If there is no vertex $v\in N(c)$ such that $|I(v)|=1$, then it is easy to check that $s(c)\le 17/6$ as at least one $I$-set has at least three codewords. Consequently, let us assume that such $v$ exists (clearly only one such vertex is possible). Without loss of generality, we may assume that $v$ is either $c-1$ or $c-3$.

      \begin{itemize}

      \item Let us assume first that $v=c-1$. Due to $c+1$, we must have $c+4\in C.$ Moreover, since $I(c+1)\neq I(c+3)$, we get $c+6\in C.$ As $I(c+2)\neq \emptyset$ (resp. $c-2$), we have $c+5\in C$ (reps. $c-5\in C$). In order to have $I(c-1)\neq I(c-3)$ we must have $c-6\in C.$ In addition, $c-7\in C$, since $I(c-2)\neq I(c-4).$ This leads to $s(c)=10/3$    and the only way to achieve this is by the pattern:
    $$
    \begin{array}{lc}
      S6 & xxxoooo\underline{x}oooxxx
    \end{array}.
    $$

    \item Suppose then that $v=c-3.$ In order to have $I(c-3)\neq I(c-1)$, we must have $c+2\in C.$ Also $c-5\in C$ because $I(c-2)$ cannot be the empty set. Moreover, $c+4\in C$ due to $I(c-1)\neq I(c+1).$ We also have $c+6\in C$ to get $I(c+1)\neq I(c+3)$. Now $s(c)=37/12$ and it comes from the pattern:
    $$
    \begin{array}{lc}
      S4 & oxoooo\underline{x}oxox*x\\
    \end{array}.
    $$
    \end{itemize}
  \end{itemize}
\end{proof}

Next we show that shifting the shares among codewords gives us the situation where the share of each vertex is (after the shifting)  less than $17/6$ or equal to 3. Moreover, the share is equal to 3 if and only if we have the case of pattern $S3.$ The share of a vertex $v \in C$ after shifting is denoted by $\ms(v)$. 
We do the shifting using the following three shifting rules and their symmetric counterparts (where the pattern is read from right to left):
\begin{figure}[htp]
\centering
\includegraphics[width=350pt]{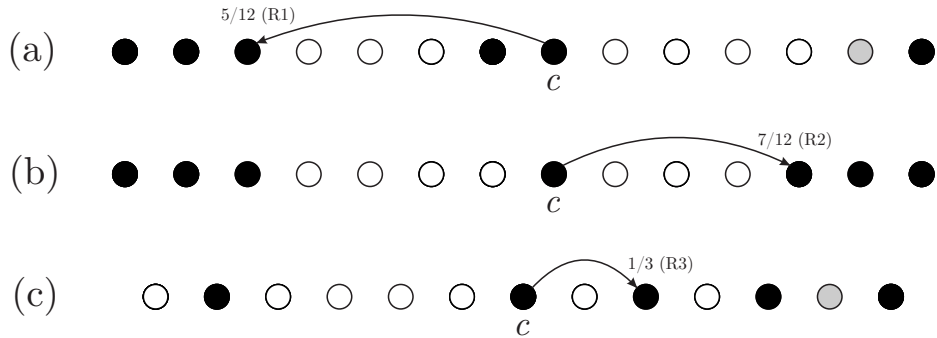}
\caption{The black nodes are codewords and white are non-codewords. The gray node can be a codeword or a non-codeword. The edges of the circulant graph are omitted in the figure.} \label{SjaL}
\end{figure}

\begin{itemize}
\item R1: The vertex $c$ gives $5/12$ units of share to $c-5$.
\item R2: The vertex $c$ gives $7/12$ units of share to $c+4$.
\item R3: The vertex $c$ gives $1/3$ units of share to $c+2$.
\end{itemize}

Notice that if a vertex receives shares by any of the rules, it cannot receive more share by another rule (including its symmetric counterpart).

\begin{proposition}\label{less3}
  Let $C$ be a locating-dominating code in $C_n(1,3)$ where $n\geq 14$. Then we have $\ms(c)\le 17/6$ for all $c\in C$ unless $c$ belongs to a pattern $S3$ when $\ms(c)=3.$
  \end{proposition}
\begin{proof}
  If in the code $C$ there are only codewords with share at most $17/6$, then there is nothing to do. Let us now consider the other cases:
  \begin{itemize}
  \item Let there be a codeword $c$ with $s(c)=3$ in the pattern $S1$.  By the rule R1, we shift 5/12 units of share to $c-5$. Notice that $s(c-5)\le 13/6.$ Consequently, we have $\ms(c)=31/12<17/6$ and $\ms(c-5)\le 31/12.$
  \item If there is a codeword $c$ with share $s(c)=37/12$ in the pattern $S4$. Using the rule  R3 we shift $1/3$ units of share to $c+2$. The share $s(c+2)\le 29/12.$ Therefore, $\ms(c)=11/4<17/6$ and $\ms(c+2)\le 11/4.$
  \item Let there be a codeword $c$ with share $s(c)=10/3$ in the pattern $S6$.
    Now R2 shifts $7/12$ units of a share to $c+4$ with $s(c+4)\le 13/6$. Consequently, $\ms(c)=11/4$ and $\ms(c+4)\le 11/4$.
  \end{itemize}
  \end{proof}

Before our main theorem on locating-dominating codes, let us give the following technical lemma.
\begin{lemma} \label{nonS3}
Let $n>17$ be an integer such that $n\equiv 3\pmod{6}$ or $n\equiv 2\pmod{3}$, and let $C$ be a locating-dominating code in $C_n(1,3)$. If there is no pattern $S3$, then $|C|>\lceil n/3\rceil$.
\end{lemma}
\begin{proof}
  By Proposition \ref{less3}, after shifting the shares and  knowing there is no $S3$ we have $\ms(c)\le 17/6$ for all $c\in C$. Hence,
  $$n=\sum_{i\in C}s(i)=\sum_{i\in C} \ms(i)\le \frac{17}{6}|C|.$$
  The proof divides now into the following cases:
  \begin{itemize}
  \item if $n=6k+3$, we have
    $ |C| \geq  2k+1 +\frac{2k+1}{17} >2k+1$
  \item and if $n=3k+2$, we have
    $ |C| \geq k+\frac{k+12}{17}$
    and as $n>17$, $k>5$, which gives $|C|>k+1$.
  \end{itemize}

\end{proof}

\begin{theorem}\label{LD-bound}
  Let $n> 17$. Then
  $$ \LD(C_n(1,3)) =
  \left\{
  \begin{array}{ll}
    \displaystyle \left\lceil\frac{n}{3}\right\rceil &
    \text{if } n\equiv 0,1,4 \pmod{6} \\
    &\\
    \displaystyle \left\lceil\frac{n}{3}\right\rceil +1 &
    \text{otherwise}
  \end{array}
  \right.
  $$
\end{theorem}
\begin{proof}
  Let $n>17$. For $n\equiv 0,1,4 \pmod{6}$ the result for locating-dominating codes is given in \cite{GNlidcn}. We  need to prove that  for $n\equiv 2,3,5\pmod{6}$ the bound $\lceil n/3\rceil$ is not attainable. On the other hand, in \cite{GNlidcn} there are constructions of cardinality $\lceil\frac{n}{3}\rceil +1$ given in these cases. We will write $n=6k+r$ with $r\in\{2,3,5\}$. Notice that for $r=2,5$, we can write $n$ in the form $3l+2$, which implies $n\equiv 2\pmod{6}$ for $l\equiv 0\pmod{2}$ and $n\equiv 5\pmod{6}$ otherwise.
  By Lemma \ref{nonS3} we know that if there is no $S3$ patterns, then the bound is not attainable.
  Assume then than there is a pattern $S3$. Note that these patterns $S3$ can overlap each other. We denote by $P6$ the pattern $xxoooo$. Therefore, we can divide overlapping $S3$-patterns into non-overlapping patterns $P6$.

  Without loss of generality, we can assume that the vertices $0,1,2,3,4,5,6,7,8,9,10,11$ form two patterns $P6$ and $s(7)=3$, that is, $0,1,6,7,12,13\in C$. As $n\not\equiv 0 \pmod{6}$, we cannot  have only patterns $P6$ in the graph. Therefore, we can assume that  there are $t$ consecutive  patterns $P6$ starting in $0$ (to the right) and that there is no pattern $P6$ on the vertices $n-6,n-5,n-4,n-3,n-2,n-1$.
  We want to prove that the sum of all the shares of codewords is strictly less than $3|C|$ for $n=6k+3$ and strictly less than $3|C|-1$ for $n=6k+2$ and $n=6k+5$. This will imply the lower bounds as we shall see.

  \begin{itemize}
  \item Let $n=6k+3$. Since there is no pattern $P6$ on the vertices $n-6,n-5,n-4,n-3,n-2,n-1$, the vertex 1 does not have the surroundings of the pattern $S3$. Therefore, $\ms(1)<3$ by Proposition~\ref{less3}.
  Hence, we have
    $$
    \begin{array}{ccc}
      \displaystyle n=\sum_{i\in C} s(i) =\sum_{i\in C} \ms(i)< 3|C|
      & \Leftrightarrow
      &  6k+3 < 3|C| \\
      & \Leftrightarrow
      & 2k+1<|C|.
    \end{array}
    $$
    Consequently, $|C|\geq 2k+2=\lceil n/3\rceil+1$.

  \item Now, let $n\equiv 2,5\pmod{6}$. It is easy to check that $s(1)\le 3$ and $s(0)\le 3.$ We will try to find such vertices, say $b$ of them, that their shares (after the shifting by the above rules) is less than $3b-1-\varepsilon$ for some $\varepsilon>0.$
    \begin{itemize}
    \item Let first $s(1)=3$. This implies that $n-5\in C$ and $n-3,n-2,n-1$ are not codewords (due to patterns~$S1$ and $S3$). If $s(0)=3$, then we get a pattern $P6$ on the vertices $n-6,n-5,n-4,n-3,n-2,n-1$, hence, as we assumed there was no such pattern, $s(0)<3$. Consequently, $n-4\in C$. In addition, $n-6\in C$ due to $I(n-1)\neq I(n-3)$. This implies that $s(0)\le 7/3.$
      The share of $n-4$ is then at most $13/6$, hence we have that the share of 0 and $n-4$ (before the shifting by the rules) drops from $2\cdot 3$ by at least $-2/3-5/6=-3/2.$ Only the rule R1 applies here and it can give to $n-4$ the amount of $5/12$. Therefore, the total drop in shares is at least $-3/2+5/12=-13/12$ (which is enough as we try to have drop of $-1-\varepsilon$). Hence,
      $$\displaystyle \sum_{i\in C} s(i)=\sum_{i\in C} \ms(i) \leq 3|C| -13/12.$$

    \item Suppose then that $s(1)<3$. In what follows, we study vertices $n-1$, $n-2$, etc., and different variants of possible codewords among them in order to find the codewords whose shares drop enough (of course, excluding the cases $s(1)=3)$.

        We start by consider separately the cases $n-1\in C$ and $n-1\notin C$.

      (i) Suppose first that $n-1\in C$. We divide further the study into two cases $n-2\in C$ and $n-2\notin C$. Let $n-2\in C$. Now $s(1)\le 2$ and $s(0)\le 2$. Therefore, the drop of the vertices 0 and 1 is $-2$ compared to $2\cdot 3$ and since no rules gives these vertices any additional share, we are done. Assume then that $n-2\notin C$. Now $s(1)\le 13/6$ and $s(0)\le 17/6.$
      Thus $s(0)+s(1)\leq 5$. If $s(0)+s(1)<5$, then the drop is $-1-\varepsilon$ and no rules give extra share to them, so we are done. If $s(0)+s(1)=5$, then we have $n-5\in C$ and $n-4,n-3,n-2\notin C$ and $s(n-1)\leq 13/6$. The drop among the vertices 0,1 and $n-1$ is altogether $-11/6$. The rules R1 and R2 can give at most $7/12$ to the vertex $n-1$ (not both at the same time), and the codewords $0$ and $1$ do not receive share according to any rule. Therefore, the total drop is at least $-11/6+7/12=-5/4.$ Hence we have
      $$\displaystyle \sum_{i\in C} \ms(i)\leq 3|C|-\frac{5}{4}.$$

      (ii) Let then $n-1\notin C$. Notice that in this case no rules give any additional share to vertices  0 and 1. In the following, we consider the cases depending on which of the vertices in $\{n-5,n-4,n-3,n-2\}$ are codewords. If there are three (or four) codewords in that set, it is easy to compute that $s(0)+s(1)\le 29/6$. Hence the drop of the vertices 0 and 1 is at least $-7/6$. Recall that the rules give no extra here. Consequently,
      $$\displaystyle \sum_{i\in C} \ms(i)\leq 3|C| -7/6.$$
      The remaining cases are listed below. Notice that since $s(1)<3$, the case where $n-5,n-4\in C$ and $n-3,n-2\notin C$ and also the case where $n-5\in C$ and $n-4,n-3,n-2\notin C$ can be excluded. Furthermore, the case $n-4\in C$ and $n-5,n-3,n-2\notin C$ is excluded because in that case $I(n-2)=I(2)=\{1\}$ and $2,n-2\notin C$ which is impossible since $C$ is locating-dominating. In the table below, we have all the other cases:
      $$
      \begin{array}{|l|c|c|c|c|c|}
        \hline
        Case & pattern & s(1) & s(0)\leq & s(p) \leq & drop \\
        \hline \hline
        \text{Case 1}
        & ooxxo\underset{0}{x}xoooox\underset{7}{x} & 5/2 & 2 & &-3/2\\
        \hline
        \text{Case 2}
        & xo\underset{p}{x}oo\underset{0}{x}xoooox\underset{7}{x} & 8/3 & 17/6 & {\color{gray} 2}  & -1/2{\color{gray} -1}\\
        \hline
        \text{Case 3}
        & xooxo\underset{0}{x}xoooox\underset{7}{x} & 8/3 & 13/6 & & -7/6\\
        \hline
        \text{Case 4}
        & oo\underset{p}{x}oo\underset{0}{x}xoooox\underset{7}{x} & 17/6 & 17/6 & {\color{gray} 17/6} & -1/3{\color{gray} -1/6}\\
        \hline
        \text{Case 5}
        & ox\underset{p}{x}oo\underset{0}{x}xoooox\underset{7}{x} & 17/6 & 13/6 & {\color{gray} 8/3} & -1{\color{gray} -1/3}\\
        \hline
        \text{Case 6}
        & ooo\underset{p}{x}o\underset{0}{x}xoooox\underset{7}{x} & 17/6 & 13/6 &{\color{gray} 8/3}  & -1{\color{gray} -1/3} \\
        \hline
        \text{Case 7}
        & oxoxo\underset{0}{x}xoooox\underset{7}{x} & 17/6 & 2 & & -7/6\\
        \hline
        \end{array}
      $$

      In all the cases except Case 4, we have a drop strictly smaller than $-1$ and the rules do not give any extra share to the vertex marked by $p$ (the vertices 0 and 1 did not get any as mentioned earlier). To examine Case 4 more carefully, we study the vertices $n-9,n-8,n-7$ and $n-6$ and codewords among them. The cases where the codewords among these four vertices are as follows $\{oooo,xoxo,xooo,ooxo,oxxo,oxoo\}$ are forbidden in a locating-dominating code. Indeed, the first four combinations give $I(n-5)=\emptyset$ and the two last ones give $I(n-6)=I(n-4)$. All the other cases are studied in the following table. In Case 4.6 we have added one more codeword, namely, the $p_0$ (which necessarily must be a codeword).
      \begin{small}
      $$
      \begin{array}{|l|c|c|c|c|c|c|}
        \hline
        Case & pattern & s(1) & s(0) & s(n-3) & \sum s(p_i)\leq  & drop \\ \hline
        \hline
        \text{Case 4.1}
        & xxxxooxoo\underset{0}{x}xooooxx
        & 17/6 & 8/3 &23/12 & & -19/12\\
        \hline
        \text{Case 4.2}
        & xxxoooxoo\underset{0}{x}xooooxx
        & 17/6 & 17/6 & 13/6& & -7/6\\
        \hline
        \text{Case 4.3}
        & xxo\underset{p}{x}ooxoo\underset{0}{x}xooooxx
        & 17/6 & 8/3 & 5/2& {\color{gray} 2}
        & -2 \\
        \hline
        \text{Case 4.4}
        & xoxxooxoo\underset{0}{x}xooooxx
        & 17/6 & 8/3 & 23/12 & & -19/12\\
        \hline
        \text{Case 4.5}
        & oxxxooxoo\underset{0}{x}xooooxx
        & 17/6 & 8/3 &2& & -3/2\\
        \hline
        \text{Case 4.6}
        & \underset{p_0}{x}\underset{p_1}{x}\underset{p_2}{x}ooooxoo\underset{0}{x}xooooxx
        & 17/6 & 17/6& 17/6&{\color{gray} 2*8/3+17/6 }
        & -1/2 {\color{gray} -5/6}\\
        \hline
        \text{Case 4.7}
        & xoo\underset{p}{x}ooxoo\underset{0}{x}xooooxx
        & 17/6 & 8/3 & 5/2& {\color{gray} 5/2}
        & -3/2\\
        \hline
        \text{Case 4.8}
        & oxo\underset{p}{x}ooxoo\underset{0}{x}xooooxx
        & 17/6 & 8/3 & 8/3& {\color{gray} 7/3}
        & -3/2\\
        \hline
        \text{Case 4.9}
        & ooxxooxoo\underset{0}{x}xooooxx
        & 17/6 & 8/3 & 2& & -3/2\\
        \hline
        \text{Case 4.10}
        & ooo\underset{p}xooxoo\underset{0}{x}xooooxx
        & 17/6 & 8/3 & 8/3& {\color{gray} 8/3}
        & -7/6 \\
        \hline
      \end{array}
      $$
      \end{small}
      Notice that the vertices 0, 1 and $n-3$ cannot receive any extra share from the rules. In addition, the codewords marked by $p$ also do not receive share by the rules. Now let us consider the special case Case 4.6. The vertices $p_2$ and $p_1$ do not receive share from the rules. If the vertex $p_0$ does not receive extra share, then the drop is enough. However, the vertex $p_0$ can get a share from the left by the rules R1 or R2. Suppose first that $p_0$ (the vertex $n-10$) receives $5/12$ units of share by the rule R1 from the vertex $n-15$. But then the vertices $n-14$ and $n-15$ belong to $C$ and thus $s(p_2)\le 8/3$, $s(p_1)\le 7/3$ and $s(p_0)\le 13/6$. Hence the new drop (taking into account the vertices 0, 1, $n-3$, $p_0$, $p_1$ and $p_2$) is at least $-7/3$. So even with the extra share the drop is enough $-7/3+5/12=-23/12$. Assume then that $p_0$ gets extra share $7/12$ by the rule R2 from $n-14$ (which belongs to $C$). Now $s(p_2)\le8/3$, $s(p_1)\le17/6$ and $s(p_0)\le 13/6.$ Consequently, the drop of the six vertices is at least $-11/6$ and with the extra share $-11/6+7/12=-5/4$, which is enough.
      \end{itemize}

    In all the cases studied above, we get that the drop of the share is strictly more than $1$. Recall that we consider the cases $n\equiv 2,5\pmod{6}$. We  write $n$ as $3k+2$, which implies $n\equiv 2\pmod{6}$ for $k\equiv 0\pmod{2}$ and $n\equiv 5\pmod{6}$ otherwise. Therefore, we have, for some $\varepsilon>0$:
    $$
    \begin{array}{ccc}
      \displaystyle n=\sum_{i\in C} \ms(i) \leq 3|C| -1 -\varepsilon
      &\Leftrightarrow
      &3k+2+1+\varepsilon \leq 3|C| \\
      &\Leftrightarrow
      & \displaystyle k+1 +\frac{\varepsilon}{3} \leq |C|\\
      &\underset{|C|\in \mathbb{N}}{\Rightarrow}
      & k+2 \leq |C|
    \end{array}
    $$
    This implies that $|C|\geq \lceil n/3\rceil+1$.
  \end{itemize}
\end{proof}

By exhaustive search we find the optimal locating-dominating codes for $n=14,15,17$, which are the remaining values after the previous theorem and the small values determined in \cite{GNlidcn}. An optimal code for $n=14,15$  is $\{0,1,2,9,10,11\}$ and for $n=16$ it is $\{0,1,7,8,12,13,14\}$.

%
%
%
%
%
%
%
%
%
%
%
%
%
%
%
%
%
%
%
%

\begin{theorem} We have
  $\LD(C_{14}(1,3))=\LD(C_{15}(1,3))=6$ and $\LD(C_{17}(1,3))=7.$
\end{theorem}

The following theorem summarizes the results of the section and gives the exact values of $\LD(C_n(1,3))$ for all lengths $n \geq 13$. The exact values of $\LD(C_n(1,3))$ for the lengths $n$ smaller than $13$ have been determined in~\cite{GNlidcn}.
\begin{theorem} \label{TheoremLDConclusion}
Let $n$ be an integer such that $n \geq 13$. If $n \equiv 0,1,4 \pmod{6}$, then $\LD(C_n(1,3)) = \lceil n/3 \rceil$, and otherwise $\LD(C_n(1,3)) = \lceil n/3 \rceil + 1$.
\end{theorem}

\medskip

In~\cite{Mlldcn}, locating-dominating codes in the circulant graphs $C_n(1,2, \ldots, r)$, where $r$ is positive integer, have been considered. However, the results regarding location-domination given in the paper have some problems as presented in the following remark.
\begin{remark} \label{RemarkManuelProblem}
In~\cite{Mlldcn}, it is claimed that $C$ is a locating-dominating code in $C_n(1,2, \ldots, r)$ if and only if the following conditions hold for all sets $\{k, k+1, \ldots, k+(2r+1)\} \ (k \in \Z_n)$ of consecutive vertices: (1) at least one of the vertices $k$ and $k+(2r+1)$ belongs to $C$ and (2) there exist at least two codewords of $C$ in $\{k, k+1, \ldots, k+(2r+1)\}$. However, this characterization is erroneous. Indeed, for example, the code $\{0, 1, 2, 5, 6, 11\}$ is locating-dominating in $C_{15}(1,2,3)$. However, the above conditions are not met when we choose $k=3$ as both $k = 3$ and $k+(2\cdot3+1) = 10$ do not belong to the code. Thus, as the given characterization is erroneous, the other results in~\cite{Mlldcn} based on that one also have problems. For more information regarding locating-dominating codes in $C_n(1,2, \ldots, r)$, which can also be viewed as power graphs of cycles, the interested reader is referred to the paper~\cite{EJLldc}.
\end{remark}


\end{document}